\documentclass[12pt]{article}
\usepackage{amssymb}
\usepackage{amsmath}
\usepackage{amsthm}
\usepackage{enumitem}
\usepackage{mathtools}
\usepackage[colorlinks=true,allcolors=black]{hyperref}
\usepackage{url}
\usepackage{float}
\usepackage{setspace}
\usepackage{tikz}
\tikzset{every picture/.style={line width=0.75pt}}
\usepackage[labelfont=bf]{caption}

\allowdisplaybreaks

\newtheorem{thm}{Theorem}
\newtheorem{prop}{Proposition}

\newtheorem{rem}{Remark}
\newtheorem{lem}{Lemma}

\newcommand{\Addresses}{{
  \bigskip
  \footnotesize

KMS:
\textsc{Department of Mathematics, Johns Hopkins University, 404 Krieger Hall, 3400 N. Charles Street, Baltimore, MD 21218, USA}\par\nopagebreak
  \textit{Email address:} \texttt{kmarsh34@jh.edu}

\medskip

MT:
\textsc{Department of Mathematics, Institute of Science Tokyo, 2-12-1 Ookayama, Meguro-ku, Tokyo 152-8551, Japan}\par\nopagebreak
  \textit{Email address:} \texttt{takadamayu11@gmail.com}

\medskip

YT:
\textsc{Department of Mathematics, Institute of Science Tokyo, 2-12-1 Ookayama, Meguro-ku, Tokyo 152-8551, Japan}\par\nopagebreak
  \textit{Email address:} \texttt{tonegawa@math.titech.ac.jp}

\medskip

MW:
\textsc{Department of Mathematics, National Taiwan Normal University, No. 88, Sec.4, Ting-Chou Road, Taipei, 116059, Taiwan (R.O.C.)}\par\nopagebreak
  \textit{Email address:} \texttt{myles.workman@math.ntnu.edu.tw}
}}

\begin{document}

\title{\Large \textbf{GRADIENT FLOW OF PHASE TRANSITIONS WITH FIXED CONTACT ANGLE}}
\author{\small KOBE MARSHALL-STEVENS, MAYU TAKADA, YOSHIHIRO TONEGAWA, \& MYLES WORKMAN}
\date{\vspace{-5ex}}
\maketitle

\begin{abstract}
    \noindent We study the gradient flow of the Allen--Cahn equation with fixed boundary contact angle in Euclidean domains for initial data with bounded energy. Under general assumptions, we establish both interior and boundary convergence properties for the solutions and associated energy measures. Under various boundary non-concentration assumptions, we show that, for almost every time, the associated limiting varifolds satisfy generalised contact angle conditions and have bounded first variation, as well as deducing that the trace of the limit of the solutions coincides with the limit of their traces. Moreover, we derive an Ilmanen type monotonicity formula, for initial data with bounded energy, valid for the associated energy measures up to the boundary.
\end{abstract}

\tableofcontents

\section{Introduction}

Throughout this paper we consider the following Allen--Cahn equation with non-linear Robin boundary conditions:
\begin{equation}\label{eqn: parabolic Allen-Cahn with contact energy introduction}
    \begin{cases}
        u \colon \overline{\Omega} \times [0, + \infty) \rightarrow \mathbb{R}, \\
        \varepsilon \partial_{t} u = \varepsilon \Delta u - \frac{W'(u)}{\varepsilon} & \text{in } \Omega ,\\
        \varepsilon (\nabla u \cdot \nu) = - \sigma' (u) & \text{on } \partial \Omega,
    \end{cases}
\end{equation}
where $\Omega \subset \mathbb{R}^n$ is a bounded domain with smooth boundary, $\varepsilon \in (0,1)$, $\nu$ is the outward pointing unit normal to $\partial \Omega$, $W$ is a smooth double-well potential with strict minima at $\pm 1$, and $\sigma$ is smooth. Equation (\ref{eqn: parabolic Allen-Cahn with contact energy introduction}) arises naturally as the gradient flow of the energy
\begin{equation}\label{eqn: energy}
    E_\varepsilon(u) = \int_\Omega \frac{\varepsilon |\nabla u|^2}{2} + \frac{W(u)}{\varepsilon} dx + \int_{\partial\Omega} \sigma(u) d\mathcal{H}^{n-1},
\end{equation}
where $\mathcal{H}^{n-1}$ denotes the $(n-1)$-dimensional Hausdorff measure on $\mathbb{R}^n$; we derive this below in the calculation resulting in (\ref{eqn: time derivative of allen cahn energy}). We note here that (\ref{eqn: parabolic Allen-Cahn with contact energy introduction}) is often sped up by a factor of $\frac{1}{\varepsilon}$ in the literature (e.g.~see \cite{I93}).

\bigskip 

The energy (\ref{eqn: energy}) was proposed in \cite{C77} (and introduced into the mathematical literature in \cite[Section 5]{G85}) as a phase transition model for fluid in a container, $\Omega$, with a contact energy between the fluid and the boundary of the container, $\partial \Omega$. The $\varepsilon \rightarrow 0$ limiting behaviour was first studied in \cite{M87} in the framework of $\Gamma$-convergence (for minimisers of (\ref{eqn: energy})) and more recently in \cite{KT18} in the framework of varifolds (for general critical points). These works establish that, given a sequence, $\varepsilon_{i} \rightarrow 0$, and critical points, $\{u_{\varepsilon_{i}}\}$ (minimisers in \cite{M87} and general critical points in \cite{KT18}), of (\ref{eqn: energy}) with uniform energy bounds along this sequence, one can extract a subsequential limiting function, $u \in L^{1} (\Omega)$, taking values in $\{ \pm 1 \}$ $\mathcal{L}^n$ a.e.~in $\Omega$, where $\mathcal{L}^n$ denotes the $n$-dimensional Lebesgue measure on $\mathbb{R}^n$.
Thus, in the $\varepsilon \rightarrow 0$ limit, the domain, $\Omega$, splits into disjoint regions, given by $\{u = \pm 1\}$. Moreover, the boundary of these regions, often referred to as the phase interface, makes fixed contact angle 
\begin{equation*}
    \theta = \arccos\left(\frac{\sigma(1) - \sigma(-1)}{c_0}\right)
\end{equation*}
with the boundary, $\partial \Omega$, in an appropriate weak sense, where
\begin{equation*}
    c_0 = \int_{-1}^1 \sqrt{2W(s)}ds.
\end{equation*}
As noted in \cite{KT17}, by a heuristic argument as well as the $\Gamma$-convergence result of \cite{M87}, one expects that the energy, (\ref{eqn: energy}), provides a phase transition approximation of the Gauss free energy; in the sense that for $\varepsilon \approx 0$ we have, for some constant $C$, that
\begin{equation*}
    c_0^{-1} E_\varepsilon(u_\varepsilon) \approx \mathcal{H}^n(\Omega \cap \{u_\varepsilon = +1\}) + \cos(\theta)\cdot\mathcal{H}^{n-1}(\partial \Omega \cap \{u_\varepsilon = +1\}) + C.
\end{equation*}
Critical points of the Gauss free energy are given by sets whose boundary inside of the domain, $\Omega$, is a minimal hypersurface (precisely a stationary varifold), meeting the boundary, $\partial \Omega$, with fixed contact angle in an appropriate weak sense. Weak notions of fixed contact angle at the boundary arise naturally in several variational problems (e.g.~for minimisers of anisotropic capillarity problems as studied in \cite{DePM15}) and were first introduced in the framework of varifolds in \cite{KT17}.

\bigskip

In this paper we study the $\varepsilon \rightarrow 0$ limiting behaviour of the gradient flow of the energy (\ref{eqn: energy}) in the framework of varifolds. One main focus is the analysis of the $\varepsilon \rightarrow 0$ limiting behaviour of the solutions to (\ref{eqn: parabolic Allen-Cahn with contact energy introduction}) and their associated energy measures under the assumption that the surface energy (the interior term in (\ref{eqn: energy})) does not concentrate onto the boundary in the limit. Physically, this corresponds to precluding the phase interface from accumulating onto the boundary along the flow, a phenomena often referred to as `wetting' in the literature; for further information on this point we refer to \cite{C77} for a physical motivation as well as the examples and references discussed in \cite[Section 8]{MT15}.

\bigskip

The interior $\varepsilon \rightarrow 0$ convergence of the gradient flow of the Allen--Cahn energy to various weak notions of the mean curvature flow have been studied extensively; we refer to \cite{I93, T03, T24} and the references therein for a more complete background. Closely related to the fixed boundary contact angle condition considered here is the gradient flow of the Allen--Cahn energy with Neumann boundary conditions, arising from (\ref{eqn: energy}) when $\sigma = 0$, as studied in \cite{MT15} for convex domains and \cite{K19} for general domains. In these works it was shown (under certain technical assumptions in the latter) that the energy measures associated to the solutions converge, as $\varepsilon \rightarrow 0$, to a integer rectifiable Brakke flow with an appropriate weak notion of free boundary (i.e.~fixed $90^\circ$ contact angle at the boundary). Though we do not pursue it here, one could hope to show the corresponding analogue of these results for (\ref{eqn: energy}) when $\sigma \neq 0$; namely, showing that the interior energy measures (defined in Subsection \ref{subsec: setting} below) associated to solutions of (\ref{eqn: parabolic Allen-Cahn with contact energy introduction}) converge to a integer rectifiable Brakke flow with an appropriate weak notion of fixed contact angle at the boundary. A fundamental tool in establishing any such convergence result is a monotonicity formula, which we derive here for the energy measures associated to solutions of (\ref{eqn: parabolic Allen-Cahn with contact energy introduction}) in Appendix \ref{sec: monotonicity}.

\bigskip 

The gradient flow of (\ref{eqn: energy}) has previously been studied in a variety of contexts. By deriving the first term in the asymptotic expansion of solutions to (\ref{eqn: parabolic Allen-Cahn with contact energy introduction}), aspects of the $\varepsilon \rightarrow 0$ limiting behaviour of the flow were first studied in \cite{OS92}. More recently, in \cite[Theorem 1 (ii)]{HL21} it is shown that, under the assumption that no energy is dropped in the limit, the gradient flow of (\ref{eqn: energy}) converges, as $\varepsilon \rightarrow 0$, to a BV solution of the mean curvature flow with fixed contact angle (in the sense of \cite[Definition 1]{HL21}). In the opposite direction, in \cite{AM22, HM22}, by assuming the existence of a smooth mean curvature flow with fixed contact angle, the existence of a solution to (\ref{eqn: parabolic Allen-Cahn with contact energy introduction}) which approximates this mean curvature flow is established and the convergence rate of the approximation is analysed.

\subsection{Notation}

We collect some notations and definitions that will be used throughout the paper:

\begin{itemize}

    \item We let $\Omega \subset \mathbb{R}^n$ be an open, bounded, connected set with smooth boundary, $\partial \Omega$, and denote by $\nu$ the outward pointing unit normal to $\partial \Omega$. Let $\kappa > 0$ denote the reciprocal of the supremum of the principal curvatures of $\partial \Omega$ (which is finite and strictly positive as $\partial\Omega$ is smooth and compact) and for $0 < r \leq \kappa$ we define the interior tubular neighbourhood of $\partial \Omega$ of size $r$ to be
    $$ N_r = \{x - \lambda \nu(x) \: | \: 0 \leq \lambda < r\}.$$
    
    \item We will denote by $\nabla^\top, \Delta^\top$ and $\mathrm{div}_{\partial\Omega}$ the tangential gradient, Laplacian and divergence operators on $\partial \Omega$ respectively.

    \item We denote by $\mathbf{G}(n,n-1)$ the space of $(n-1)$-dimensional subspaces of $\mathbb{R}^n$, and identify $S \in \mathbf{G}(n,n-1)$ with the orthogonal projection of $\mathbb{R}^n$ onto $S$ and its matrix representation. For $a \in \mathbb{R}^n$ we denote by $a \otimes a \in \mathrm{Hom}(\mathbb{R}^n; \mathbb{R}^n)$ the matrix with entries $a_i a_j$ (with $1 \leq i,j \leq n$). Writing $I$ for the identity matrix we have, for each unit vector $a \in \mathbb{R}^n$, $I - a \otimes a \in \mathbf{G}(n, n-1)$. For $X \subset \mathbb{R}^n$ either open or compact we set $\mathbf{G}_{n-1}(X) = X \times \mathbf{G}(n,n-1)$. 

    \item For a Radon measure $\mu$ on a measure space $X$ and $\varphi \in C_c(X)$ we write $\mu(\varphi) = \int \varphi \, d\mu$ and let $\mathrm{spt} (\mu)$ denote the support of $\mu$. A general $(n-1)$-varifold (hereafter simply varifold) on $X$ is a Radon measure on $\mathbf{G}_{n-1}(X)$, and we denote the set of all such varifolds as $\mathbf{V}_{n-1}(X)$. For $V \in \mathbf{V}_{n-1}(X)$ we will write $||V||$ for the weight measure of $V$, i.e.~for each $\phi \in C_c(X)$ we have
    $$||V||(\phi) = \int_{\mathbf{G}_{n-1}(X)} \phi(x) \, dV(x,S),$$
    and $\delta V$ for the first variation of $V$, i.e.~for each $g \in C^1_c(X; \mathbb{R}^n)$ we have
    $$\delta V (g) = \int_{\mathbf{G}_{n-1}(X)} \nabla g(x) \cdot S \, dV(x,S).$$

\end{itemize}

\subsection{Setting}\label{subsec: setting}

We first impose the following two assumptions on $W$ and $\sigma$, these will be fixed throughout the paper:

\begin{enumerate}\MakeLinkTarget{}\label{asmp: on W}
    \item[(A1)] $W \in C^{\infty}(\mathbb{R})$ is a non-negative double-well potential with non-degenerate minima at $\pm 1$, unique local maximum in $(-1,1)$, and for some $\gamma \in (0,1)$ we have $W''(s) > 0$ for each $|s| > \gamma$.
    
    \item[(A2)]\MakeLinkTarget{}\label{asmp: on sigma} $\sigma \in C^{\infty} (\mathbb{R})$ and there exists a $c_1 \in [0, 1)$ such that for each $s \in \mathbb{R}$ we have
    \begin{equation*}
        |\sigma' (s)| \leq c_1 \sqrt{2W (s)}.
    \end{equation*}
    As the angle condition is determined by the difference $(\sigma(1) - \sigma(-1))$ and $c_0$, without loss of generality we set $\sigma(-1) = 0$ throughout.
 \end{enumerate}

\begin{rem} A typical example of a potential in \hyperref[asmp: on W]{(A1)} is $W(s) = \frac{1}{4}(1-s^2)^2$; one can then prescribe a fixed contact angle, $\theta \in (0,\frac{\pi}{2}]$, based on this potential by choosing $\sigma(s) = \cos(\theta) \int_{-1}^{s} \sqrt{2W(r)} \, dr$ to satisfy \hyperref[asmp: on sigma]{(A2)}. However, the above assumptions placed on $W$ and $\sigma$ above are rather general and allow for flexibility in applying our results. For instance, no specific choices of $W$ and $\sigma$ are required in order to approximate, in the $\varepsilon \rightarrow 0$ limit, hypersurfaces making fixed contact angle $\theta$ with $\partial \Omega$.
\end{rem}

\begin{rem}
    As remarked in \cite{KT18}, assumption \hyperref[asmp: on sigma]{(A2)} ensures that $$|\sigma(1)| \leq \int_{-1}^1 |\sigma'(s)| \, ds \leq c_1 \int_{-1}^1 \sqrt{2W(s)} \, ds < c_0,$$
    which physically corresponds to the contact energy density, $|\sigma(1)|$, of the interface, formed by the region $\{u \approx 1\}$, with $\partial \Omega$ being strictly smaller than the interior energy density, $c_0$, of the interface inside of $\Omega$. As $|\sigma(1)| \nearrow c_0$ we thus expect the `perfect wetting' phenomena described in \cite{C77}.
\end{rem}

Given $\varepsilon \in (0,1)$, $u \in C^{\infty} (\overline{\Omega} \times [0, \infty))$ solving (\ref{eqn: parabolic Allen-Cahn with contact energy introduction}), and $t \geq 0$ we define, for each $\varphi \in C_{c} (\mathbb{R}^{n})$, the associated Radon measures 
\begin{equation}\label{eqn: measure definitions}
    \begin{cases}
        \mu_{t, 1}^{\varepsilon}(\varphi) =  \int_{\Omega} \varphi \left(\frac{\varepsilon |\nabla u ( \, \cdot \, , \, t)|^2}{2} + \frac{W(u ( \, \cdot \, , \, t))}{\varepsilon} \right) \, dx,\\
        \mu_{t, 2}^{\varepsilon}(\varphi) = \int_{\partial\Omega } \varphi \, \sigma(u ( \, \cdot \, , t)) d\mathcal{H}^{n-1},\\
         \mu_{t}^{\varepsilon} (\varphi) = \mu_{t, 1}^{\varepsilon} (\varphi) + \mu_{t, 2}^{\varepsilon} (\varphi),
    \end{cases}
\end{equation}
and, for each $\phi \in C_{c} ( \mathbf{G}_{n - 1} (\mathbb{R}^{n}) )$, we define the associated varifolds
\begin{equation}\label{eqn: varifold definitions}
    \begin{cases}
        V_{t, 1}^{\varepsilon} (\phi) = \int_{\Omega \cap \{ \nabla u \not= 0\}} \phi \left( x, I - \frac{\nabla u}{|\nabla u|} \otimes \frac{\nabla u}{|\nabla u|} \right) d\mu_{t,1}^\varepsilon, \\
        V_{t, 2}^{\varepsilon} (\phi) = \int_{\partial\Omega } \phi (x, T_{x} \partial \Omega) \, d\mu_{t,2}^\varepsilon, \\
        V_{t}^{\varepsilon} (\phi) = V_{t, 1}^{\varepsilon} (\phi) + V_{t, 2}^{\varepsilon} (\phi).
    \end{cases}
\end{equation}
With the above definitions we have $||V_t^\varepsilon|| = \mu_t^\varepsilon$ and $||V_{t,j}^\varepsilon|| = \mu_{t,j}^\varepsilon$ for $j = 1,2$. Furthermore, we define the discrepancy functions, $\xi_t^\varepsilon$,  by setting
\begin{equation*}
   \xi_{t}^{\varepsilon} = \frac{\varepsilon |\nabla u( \, \cdot \,, \, t)|^{2}}{2} - \frac{W (u( \, \cdot \,, \, t))}{\varepsilon},
\end{equation*}
and write $d\xi_t^\varepsilon = \xi_t^\varepsilon d\mathcal{L}^n \lfloor_\Omega$. As we are concerned with the $\varepsilon \rightarrow 0$ limiting behaviour of solutions to (\ref{eqn: parabolic Allen-Cahn with contact energy introduction}) we will often make two further assumptions throughout the paper:

\begin{enumerate}
   \item[(A3)]\MakeLinkTarget{}\label{asmp: gradient flow} For a sequence $\{\varepsilon_{i}\} \subset (0,1)$, with $\varepsilon_{i} \rightarrow 0$, there exist $\{ u_{i} \} \subset C^{\infty} (\overline{\Omega} \times [0, \infty))$ such that 
    \begin{equation}\label{eqn: parabolic Allen-Cahn with contact energy}
    \begin{cases}
        u_{i} \colon \overline{\Omega} \times [0, + \infty) \rightarrow \mathbb{R}, \\
        \varepsilon_{i} \partial_{t} u_{i} = \varepsilon_{i} \Delta u_{i} - \frac{W'(u_{i})}{\varepsilon_{i}} & \text{in }  \Omega, \\
        \varepsilon_{i} (\nabla u_{i} \cdot \nu) = - \sigma' (u_{i}) & \text{on } \partial \Omega,
    \end{cases}
\end{equation}
where $\nu$ is the outward pointing unit normal to $\partial \Omega$.
\end{enumerate}
\begin{rem}
    We will often consider further subsequences, $\{\varepsilon_{i_j}\} \subset \{\varepsilon_i\}$, and when doing so, for each $t \geq 0$, write $u_j$ in place of $u_{i_j}$, $\mu_t^j, \mu_{t,1}^j, \mu_{t,2}^j$ in place of $\mu_t^{\varepsilon_{i_j}}, \mu_{t,1}^{\varepsilon_{i_j}}, \mu_{t,2}^{\varepsilon_{i_j}}$, and $V_t^j, V_{t,1}^j, V_{t,2}^j$ in place of $V_t^{\varepsilon_{i_j}}, V_{t,1}^{\varepsilon_{i_j}}, V_{t,2}^{\varepsilon_{i_j}}$ to denote the solutions of (\ref{eqn: parabolic Allen-Cahn with contact energy}), the Radon measures, and the varifolds associated with such a subsequence of solutions respectively.
\end{rem}

\begin{enumerate}
\item[(A4)]\MakeLinkTarget{}\label{asmp: uniform bounds} For the solutions of (\ref{eqn: parabolic Allen-Cahn with contact energy}) in \hyperref[asmp: gradient flow]{(A3)} there exists $E_{0} > 0$ such that 
\begin{equation*}
   \begin{cases}
       \sup_{i} E_{\varepsilon_{i}} (u_{i} ( \, \cdot \, , 0)) \leq E_{0},\\
       \sup_{i}||u_i(\, \cdot \, , 0)||_{L^\infty(\overline{\Omega})} \leq 1.
    \end{cases}
\end{equation*}
\end{enumerate}

Under the assumptions \hyperref[asmp: on W]{(A1)}-\hyperref[asmp: uniform bounds]{(A4)}, each $u_{i}$ is a solution to the gradient flow of (\ref{eqn: energy}) with initial data given by $u_{i} (\, \cdot \, , 0)$. To see this, observe that for each $\varphi \in C^{\infty} (\overline{\Omega})$ and $t \geq 0$ we have
\begin{equation*}
    \int_{\Omega} \varepsilon_{i} \, \partial_{t} u_{i} \, \varphi \, dx = - \int_{\Omega} \left( \varepsilon_{i} \nabla u_{i} \cdot \nabla \varphi + \frac{W' (u_{i})}{\varepsilon_{i}} \, \varphi \right) \, dx - \int_{\partial \Omega} \sigma' (u_{i}) \, \varphi \, d \mathcal{H}^{n - 1} = - \delta E_{\varepsilon_{i}} (u_{i} ( \, \cdot \, , t)) (\varphi),
\end{equation*}
and therefore
\begin{equation}\label{eqn: time derivative of allen cahn energy}
    \frac{d}{dt} E_{\varepsilon_{i}} (u_{i} ( \, \cdot \, , t)) =  \delta E_{\varepsilon_{i}} (u_{i} ( \, \cdot \, , t)) (\partial_{t} u_{i} ( \, \cdot \, , t)) = - \int_{\Omega} \varepsilon_{i} \, (\partial_{t} u_{i})^{2} \, dx \leq 0.
\end{equation}
Combining this fact with the energy bound in \hyperref[asmp: uniform bounds]{(A4)} we have
\begin{equation}\label{eqn: uniform t and i energy bounds}
    \sup_{i} \sup_{t \geq 0} E_{\varepsilon_{i}} (u_{i} ( \, \cdot \, , t)) \leq E_{0}.
\end{equation}
The uniform energy bounds imposed by (\ref{eqn: uniform t and i energy bounds}) ensure uniform (i.e.~independent of $\varepsilon$) mass bounds along the sequences of Radon measures and varifolds as defined in (\ref{eqn: measure definitions}) and (\ref{eqn: varifold definitions}) respectively. Consequently, due to the weak compactness of Radon measures, up to a subsequence, one obtains limits for these sequences of Radon measures and varifolds. We will denote the limiting Radon measures of the sequences $\{\mu_t^{\varepsilon_i}\}$ and $\{\mu_{t,j}^{\varepsilon_i}\}$ for $j = 1,2$ by $\mu_t$ and $\mu_{t,j}$ for $j = 1,2$ respectively. Similarly, we will denote the limiting varifolds of the sequences $\{V_t^{\varepsilon_i}\}$ and $\{V_{t,j}^{\varepsilon_i}\}$ for $j = 1,2$ by $V_t$ and $V_{t,j}$ for $j = 1,2$ respectively.

\begin{rem}\label{rem: a.e. convergence of the discrepancy}
     In \cite{T03} it was established that for $\mathcal{L}^1$ a.e.~$t \geq 0$ we have, under assumptions \hyperref[asmp: on W]{(A1)}-\hyperref[asmp: uniform bounds]{(A4)}, that potentially up to a further subsequence depending on $t$ (not relabelled) $|\xi_t^{\varepsilon_i}|d\mathcal{L}^n\lfloor_\Omega\rightharpoonup \xi_t$ for some Radon measure $\xi_t$ with $\mathrm{spt} (\xi_t) \subset \partial \Omega$ and $\xi_t \ll ||V_{t,1}|| \lfloor_{\partial\Omega}$ (i.e.~with $V^{\varepsilon_i}_{t,1} \rightharpoonup V_{t,1}$).
\end{rem}

\subsection{Main results}
Subject to our general assumptions we obtain:
\clearpage
\begin{thm}\label{thm: 1}
    Under assumptions \hyperref[asmp: on W]{(A1)}-\hyperref[asmp: uniform bounds]{(A4)}, the following holds:
    \begin{enumerate}
        \item There exist Radon measures, $\{\mu_t\}_{t \geq 0}$, and a subsequence, $\{\varepsilon_i\} \subset (0,1)$ (denoted by the same index), such that
        $$\mu_t^{\varepsilon_{i}} \rightharpoonup \mu_t \text{ on } \overline{\Omega}.$$

        \item For $\mathcal{L}^1$ a.e.~$t \geq 0$ there exists a function, $u(\,\cdot\,,t) \in BV(\Omega)$, and a subsequence, $\{\varepsilon_i\} \subset (0,1)$ (denoted by the same index), such that, for $\mathcal{L}^1$ a.e.~$t \geq 0$, we have
        $$\begin{cases}
            u_i(\,\cdot\,,t) \lfloor_{\Omega} \rightarrow u(\,\cdot\,,t) & \text{in } L^1(\Omega),\\
            u(\,\cdot\,,t) = \pm 1 & \mathcal{L}^n \text{ a.e.~on } \Omega,
        \end{cases}$$
        where $u_i(\,\cdot\,,t)\lfloor_\Omega$ denotes the restriction of $u_i(\,\cdot\,,t)$ to $\Omega$.

        \item For $\mathcal{L}^1$ a.e.~$t \geq 0$ there exists a subsequence, $\{\varepsilon_{i_j}\} 
        \subset \{\varepsilon_i\}$, dependent on $t \geq 0$, and a function, $\tilde{u}(\, \cdot \,, t) \in BV(\partial \Omega)$, such that
        $$\begin{cases}
        u_j(\,\cdot\,,t)\lfloor_{\partial\Omega} \rightarrow \tilde{u}(\,\cdot\,,t) & \text{in } L^1(\partial\Omega),\\
        \tilde{u}(\,\cdot\,,t) = \pm 1 & \mathcal{H}^{n-1} \text{ a.e.~on } \partial \Omega,
        \end{cases}$$
      where $u_j(\,\cdot\,,t)\lfloor_{\partial\Omega}$ denotes the restriction of $u_j(\,\cdot\,,t)$ to $\partial\Omega$.
    \end{enumerate}
\end{thm}

Theorem \ref{thm: 1} tells us that, under assumptions \hyperref[asmp: on W]{(A1)}-\hyperref[asmp: uniform bounds]{(A4)}, there exists a family of unique limiting measures, $\{\mu_t\}_{t \geq 0}$, for all $t \geq 0$, a unique limiting interior function, $u(\, \cdot \, , t)$, for $\mathcal{L}^1$ a.e.~$t \geq 0$, and a limiting boundary function, $\tilde{u}(\, \cdot \,, t)$, dependent on the choice of $\mathcal{L}^1$ a.e.~$t \geq 0$. 

\bigskip

We will now impose the assumption that the discrepancy vanishes in the limit, or more precisely that it does not concentrate on the boundary, namely:
\begin{enumerate}
    \item[(A5)] \MakeLinkTarget{}\label{asmp: discrepancy}
        For the sequence $\{\varepsilon_i\} \subset 
        (0,1)$ and $\mathcal{L}^1$ a.e.~$t \geq 0$ it holds that $$|\xi_t^{\varepsilon_i}|d\mathcal{L}^n \lfloor_\Omega\rightharpoonup 0.$$
        In other words, in addition to Remark \ref{rem: a.e. convergence of the discrepancy}, we have that $\xi_t\lfloor_{\partial\Omega} = 0$ for $\mathcal{L}^1$ a.e.~$t \geq 0$.
\end{enumerate}

Subject to the discrepancy non-concentration assumption we obtain:

\begin{thm}\label{thm: 2}
    Under assumptions \hyperref[asmp: on W]{(A1)}-\hyperref[asmp: discrepancy]{(A5)}, the following holds:

    \begin{enumerate}
        \item For $\mathcal{L}^1$ a.e.~$t \geq 0$ there exists a subsequence,  $\{\varepsilon_{i_{j}}\} \subset \{\varepsilon_{i}\}$, dependent on $t \geq 0$, a limiting varifold, $V_{t, 1} \in \mathbf{V}_{n-1}(\overline{\Omega})$ (i.e.~with $V_{t, 1}^j \rightharpoonup V_{t, 1}$), and a $||V_{t, 1}||$ measurable vector field, $H_t \in L^2_{||V_{t,1}||}(\overline{\Omega}; \mathbb{R}^{n})$ (i.e.~$L^2$ with respect to $||V_{t,1}||$), such that, for each $g \in C_{c}^{1}(\mathbb{R}^n; \mathbb{R}^{n})$ with $(g \cdot \nu) = 0$ on $\partial \Omega$, we have
    \begin{equation*}
        \delta V_{t, 1} (g) + \sigma(1)\int_{\{ \tilde{u} ( \, \cdot \, , \, t) = +1\}} \mathrm{div}_{\partial \Omega} (g)\, d \mathcal{H}^{n - 1} = - \int_{\overline{\Omega}} g \cdot H_t \, d ||V_{t,1}||,
    \end{equation*}
    where $\tilde{u} ( \, \cdot \, , t)$ is as in Theorem \ref{thm: 1} part 3.
   
    \item For $\mathcal{L}^1$ a.e.~$t \geq 0$ there exists a subsequence, $\{\varepsilon_{i_{j}}\} \subset \{\varepsilon_{i}\}$, dependent on $t \geq 0$, and a limiting varifold, $V_t \in \mathbf{V}_{n-1}(\overline{\Omega})$ (i.e.~with $V_t^j \rightharpoonup V_t$), such that
    \begin{equation*}
        ||\delta V_t||(\overline{\Omega}) < \infty,
    \end{equation*}
    and for each $T > 0$ we have
    \begin{equation*}
        \int_0^T ||\delta V_t|| (\overline{\Omega}) \, dt < \infty.
    \end{equation*}
    \end{enumerate}
\end{thm}
Theorem \ref{thm: 2} tells us that, under assumptions \hyperref[asmp: on W]{(A1)}-\hyperref[asmp: discrepancy]{(A5)}, for $\mathcal{L}^1$ a.e.~$t \geq 0$ there is an interior limiting varifold, $V_{t, 1}$, which is a varifold of fixed contact angle, given by $\theta = \arccos(\frac{\sigma (1)}{c_0})$, with respect to the set $\{\tilde{u} ( \, \cdot \, , t) = 1\} \subset \partial \Omega$, and that there is a limiting varifold, $V_t$, with bounded first variation which is also integrable in time.

\begin{rem}\label{rem: differing definition}
    We note that our definition of  varifold with fixed contact angle slightly differs from that of \cite[Definition 3.1]{KT17} as we allow for our varifold to have measure on the boundary $\partial \Omega$. 
\end{rem}

\begin{rem}
One expects that assumption \hyperref[asmp: discrepancy]{(A5)} can be deduced explicitly, as was done for gradient flow of the Allen--Cahn energy in the Neumann case in \cite{MT15}, from which the results of Theorem \ref{thm: 2} would hold with assumptions \hyperref[asmp: on W]{(A1)}-\hyperref[asmp: uniform bounds]{(A4)} only; we have not yet been able to show this.
\end{rem}

Under a further assumption that none of the interior measure, $\mu_{t,1}^\varepsilon = ||V_{t,1}^\varepsilon||$, accumulates on the boundary in the limit, we are able to deduce stronger conclusions about the behaviour of the limiting functions and the first variation of the limiting varifolds. Precisely, we assume that:

\begin{enumerate}
    \item[(A6)]\MakeLinkTarget{}\label{asmp: measure} For $\mathcal{L}^1$ a.e.~$t \geq 0$ and for each $s > 0$, there exists a $\delta > 0$, with
    \begin{equation*}
        \limsup_{i \rightarrow \infty}|| V_{t, 1}^{\varepsilon_i} || (N_{\delta}) < s,
    \end{equation*}
    where $N_{\delta}$ denotes the interior tubular neighbourhood of $\partial \Omega$ of size $\delta$.
\end{enumerate}

\begin{rem}
As $\xi_t \ll ||V_{t,1}||\lfloor_{\partial \Omega}$ (see Remark \ref{rem: a.e. convergence of the discrepancy}), assumption \hyperref[asmp: measure]{(A6)} is stronger than assumption \hyperref[asmp: discrepancy]{(A5)}.
\end{rem}

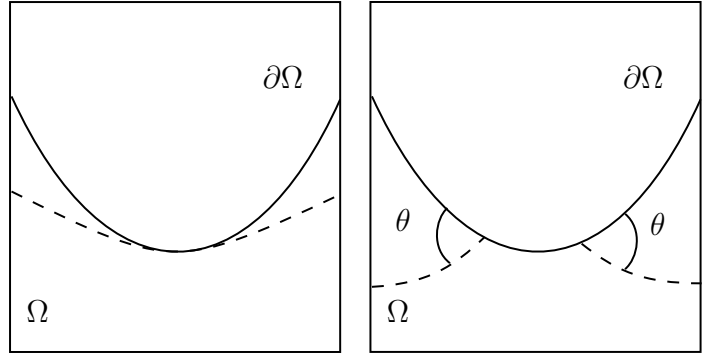
\begin{figure}[H]
\centering
\begin{tikzpicture}[x=0.75pt,y=0.75pt,yscale=-1.5,xscale=1.5]

\draw    (189.5,91.25) .. controls (221.5,161.25) and (269.5,160.25) .. (300,92.25) ;
\draw  [dash pattern={on 4.5pt off 4.5pt}]  (189.5,123.25) .. controls (243,149.75) and (245,150.25) .. (300,124.25) ;
\draw   (189,59.5) -- (300,59.5) -- (300,177.25) -- (189,177.25) -- cycle ;

\draw (273,79.9) node [anchor=north west][inner sep=0.75pt]    {$\partial \Omega $};
\draw (193.5,158.9) node [anchor=north west][inner sep=0.75pt]    {$\Omega $};

\end{tikzpicture}
\hspace{1mm}
\begin{tikzpicture}[x=0.75pt,y=0.75pt,yscale=-1.5,xscale=1.5]

\draw    (209.5,111.25) .. controls (241.5,181.25) and (289.5,180.25) .. (320,112.25) ;
\draw   (209,79.5) -- (320,79.5) -- (320,197.25) -- (209,197.25) -- cycle ;
\draw  [dash pattern={on 4.5pt off 4.5pt}]  (280,160.83) .. controls (291.4,169.8) and (300.2,174.2) .. (320.33,174.17) ;
\draw  [dash pattern={on 4.5pt off 4.5pt}]  (209.5,175.35) .. controls (228.6,174.2) and (236.68,168.7) .. (247.4,158.6) ;
\draw    (294.64,150.8) .. controls (300.77,156) and (298.91,165.8) .. (294.6,169.4) ;
\draw    (235.74,167.39) .. controls (229.36,162.5) and (230.74,152.61) .. (234.86,148.81) ;

\draw (293,99.9) node [anchor=north west][inner sep=0.75pt]    {$\partial \Omega $};
\draw (213.5,178.9) node [anchor=north west][inner sep=0.75pt]    {$\Omega $};
\draw (216.4,148.2) node [anchor=north west][inner sep=0.75pt]    {$\theta $};
\draw (302,149.4) node [anchor=north west][inner sep=0.75pt]    {$\theta $};
\end{tikzpicture}
\captionsetup{justification=justified,margin=1cm}
\caption{In both graphics the interior solid curves depict a portion of $\partial \Omega$, and the dashed curves depict smooth boundaries of interface regions in the interior of $\Omega$. The left-hand graphic depicts a positive time along the flow at which the interface comes into tangential contact with $\partial \Omega$, i.e.~the interface `wetting'. The right-hand graphic depicts the interface at any later positive time after which the `popping' of the interface has occurred, subject to the assumption that the interior measures do not concentrate on $\partial \Omega$.} \label{fig: measure non-concentration}
\end{figure}

Intuitively, this assumption ensures that the interior portions of any limiting varifold, $V_{t,1}$, do not tangentially touch, or lie within, the boundary along the flow. Indeed, since the weight measures of the varifolds, $V_{t,1}^\varepsilon$, are expected to behave like hypersurface measures of moving phase boundaries, this ensures that if such a hypersurface were to touch the boundary from the interior at some point in time, then the hypersurface would split at this point of the boundary into two distinct parts instantaneously. Physically, this precludes any `wetting' occurring along the flow, as discussed above, and corresponds to a `popping' of the interface upon tangential contact with the boundary; see Figure \ref{fig: measure non-concentration}.

\begin{rem}\label{rem: restriction of first variation}
    In view of Theorem \ref{thm: 2} part 1, for each interior limiting varifold, $V_{t,1} \in \mathbf{V}_{n-1}(\overline{\Omega})$ (i.e.~with $V_{t,1}^j \rightharpoonup V_{t,1}$), for $g \in C_c^1(\mathbb{R}^n;\mathbb{R}^n)$ with $(g \cdot \nu) = 0$ on $\partial \Omega$ we let
    \begin{equation*}
        (\delta V_{t,1})\lfloor_{A}(g) = -\sigma(1) \int_{\{\tilde{u}(\, \cdot \, , \, t) = +1\} \cap A} \mathrm{div}_{\partial\Omega} (g)\, d \mathcal{H}^{n - 1} - \int_{\overline{\Omega} \cap A} g \cdot H_t \, d||V_{t,1}||
    \end{equation*}
    denote the restriction of the first variation to a measurable set $A \subset \mathbb{R}^n$.
\end{rem}

Subject to the measure non-concentration assumption we obtain:

\begin{thm}\label{thm: 3} Under assumptions \hyperref[asmp: on W]{(A1)}-\hyperref[asmp: measure]{(A6)}, the following holds:
\begin{enumerate}
    \item For $\mathcal{L}^1$ a.e.~$t \geq 0$ there exists a subsequence, $\{\varepsilon_{i_j}\} \subset \{\varepsilon_i\}$, dependent on $t \geq 0$, such that
\begin{equation*}
\begin{cases}
    u_{j}(\,\cdot\,,t) \lfloor_{\partial\Omega} \rightarrow Tu(\,\cdot\,,t) & \text{in } L^{1} (\partial \Omega), \\
    Tu(\,\cdot\,,t) = \pm 1 & \mathcal{H}^{n-1} \text{ a.e.~on } \partial \Omega,
\end{cases}
\end{equation*}
where $Tu(\,\cdot\,,t) \in L^1(\partial\Omega)$ arises as the trace of $u(\,\cdot\,,t)$ as in Theorem \ref{thm: 1} part 2. 
Moreover, we have that $Tu ( \, \cdot \, , t) \in BV (\partial \Omega)$.

    \item For $\mathcal{L}^1$ a.e.~$t \geq 0$ there exists a subsequence, $\{\varepsilon_{i_{j}}\} \subset \{\varepsilon_{i}\}$, dependent on $t \geq 0$, and a limiting varifold, $V_{t,1} \in \mathbf{V}_{n-1}(\overline{\Omega})$ (i.e.~with $V_{t,1}^j \rightharpoonup V_{t,1}$), such that for each $g \in C_c^1(\mathbb{R}^n;\mathbb{R}^n)$ with $(g \cdot \nu) = 0$ on $\partial \Omega$, we have 
    \begin{equation}
        (\delta V_{t,1})\lfloor_{\partial\Omega}(g) = -\sigma(1) \int_{\partial^*\{Tu(\,\cdot\,,t) = +1\}} g \cdot \vec{n}_{\{Tu(\,\cdot\,,t) = +1\}} \, d\mathcal{H}^{n-2},
    \end{equation}
    where $\vec{n}_{\{Tu(\,\cdot\, , \, t) = +1\}}$ is the measure theoretic outward pointing normal to $\partial^*\{Tu(\,\cdot\,, \, t) = +1\}$.
\end{enumerate}
\end{thm}

Theorem \ref{thm: 3} tells us that, under assumptions \hyperref[asmp: on W]{(A1)}-\hyperref[asmp: measure]{(A6)}, for $\mathcal{L}^1$ a.e.~$t \geq 0$, given the functions $u ( \, \cdot \, , t)$ and $\tilde{u} ( \, \cdot \, , t)$ from Theorem \ref{thm: 1}, we have $T u = \tilde{u}$ on $\partial \Omega$, where $Tu$ is the $BV$ trace of $u$ on $\partial \Omega$ (e.g.~see \cite[Theorem 5.6]{EG15}), and that there is an interior limiting varifold, $V_{t,1}$, which makes contact angle $\theta = \arccos(\frac{\sigma (1)}{c_0})$ (in the sense of \cite[Definition 3.1]{KT17}, as opposed to in Theorem \ref{thm: 2} part 1; see Remark \ref{rem: differing definition}) with respect to the set $\{Tu(\,\cdot\,, t) = +1\}$; see Figure \ref{fig: contact angle}.

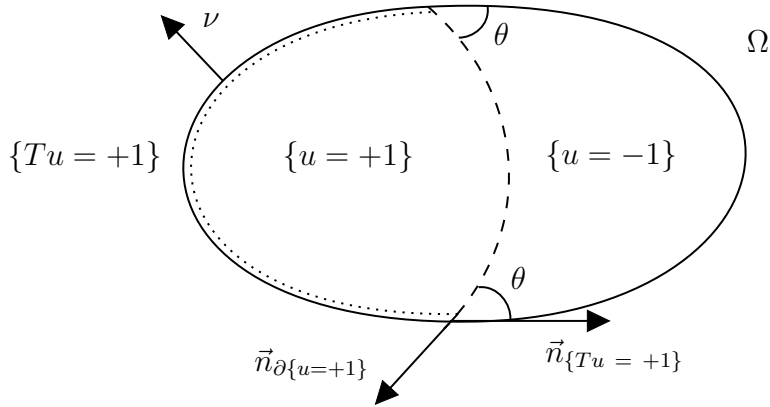
\begin{figure}[H]
\centering
\begin{tikzpicture}[x=0.75pt,y=0.75pt,yscale=-1.25,xscale=1.25]
\draw  [dash pattern={on 4.5pt off 4.5pt}]  (333,103.35) .. controls (373.4,138.2) and (377.2,195.25) .. (339.4,233.4) ;
\draw    (356,230.35) .. controls (193.4,235.8) and (201.8,108.6) .. (333,103.35) ;
\draw    (333,103.35) .. controls (509,94.6) and (491,225.4) .. (356,230.35) ;
\draw    (345,227.4) -- (313.82,261.58) ;
\draw [shift={(311.8,263.8)}, rotate = 312.37] [fill={rgb, 255:red, 0; green, 0; blue, 0 }  ][line width=0.08]  [draw opacity=0] (8.93,-4.29) -- (0,0) -- (8.93,4.29) -- cycle    ;
\draw    (341.8,230.2) -- (403.2,230.2) ;
\draw [shift={(406.2,230.2)}, rotate = 180] [fill={rgb, 255:red, 0; green, 0; blue, 0 }  ][line width=0.08]  [draw opacity=0] (8.93,-4.29) -- (0,0) -- (8.93,4.29) -- cycle    ;
\draw    (250.22,133.23) -- (227.46,109.18) ;
\draw [shift={(225.4,107)}, rotate = 46.59] [fill={rgb, 255:red, 0; green, 0; blue, 0 }  ][line width=0.08]  [draw opacity=0] (8.93,-4.29) -- (0,0) -- (8.93,4.29) -- cycle    ;
\draw  [dash pattern={on 0.84pt off 2.51pt}]  (345,227.4) .. controls (201,229.75) and (205,110.25) .. (336.2,105.4) ;
\draw    (353.4,216.6) .. controls (361.4,215.8) and (366.6,223.8) .. (366.2,229.4) ;
\draw    (357,103.02) .. controls (358.17,110.97) and (351.01,115.94) .. (345.4,115.8) ;

\draw (460.4,111.6) node [anchor=north west][inner sep=0.75pt]    {$\partial\Omega $};
\draw (379.2,236.6) node [anchor=north west][inner sep=0.75pt]    {$\vec{n}_{\{Tu\ =\ +1\}}$};
\draw (162.3,156.6) node [anchor=north west][inner sep=0.75pt]    {$\{Tu=+1\}$};
\draw (273.2,156.6) node [anchor=north west][inner sep=0.75pt]    {$\{u=+1\}$};
\draw (379.2,156.2) node [anchor=north west][inner sep=0.75pt]    {$\{u=-1\}$};
\draw (365.2,207.2) node [anchor=north west][inner sep=0.75pt]    {$\theta $};
\draw (240.8,104.7) node [anchor=north west][inner sep=0.75pt]    {$\nu $};
\draw (261.8,239.9) node [anchor=north west][inner sep=0.75pt]    {$\vec{n}_{\partial \{u=+1\}}$};
\draw (358.4,108.8) node [anchor=north west][inner sep=0.75pt]    {$\theta $};

\end{tikzpicture}
\captionsetup{justification=justified,margin=1cm}
\caption{The thicker dashed line in the graphic depicts a smooth boundary of the region $\{u = +1\}$ for the limiting function $u$ (at some fixed time) which makes contact angle $\theta$ with respect to the set $\{Tu = + 1\}$, depicted with the thin dashed line, in $\partial \Omega$. The angle $\theta$ corresponds to the angle formed between the unit inward pointing co-normal, $-\vec{n}_{\partial\{u = +1\}}$, to the boundary of $\{u = +1\}$ in $\partial \Omega$, and the outward pointing unit normal, $\vec{n}_{\{Tu = +1\}}$, to $\{Tu = +1\}$, which is tangent to $\partial \Omega$.}
\label{fig: contact angle}
\end{figure}

\begin{rem}\label{rem: canonical measures under non-concentration}
    Under assumption \hyperref[asmp: measure]{(A6)}, for $\mathcal{L}^1$ a.e.~$t \geq 0$ we have $||V_{t,1}^{\varepsilon_i}|| \rightharpoonup \mu_t\lfloor_{\Omega}$,  $||V_{t,2}^{\varepsilon_i}|| \rightharpoonup \mu_t\lfloor_{\partial\Omega}$, and, by \cite{I93,T03}, that the measures $\mu_t\lfloor_{\Omega}$ are rectifiable. For such $\mathcal{L}^1$ a.e.~$t \geq 0$ we have a unique rectifiable interior limiting varifold, $V_{t,1} \in \mathbf{V}_{n-1}(\Omega)$, with $V_{t,1}^{\varepsilon_i} \rightharpoonup V_{t,1}$ and $||V_{t,1}|| = \mu_t\lfloor_{\Omega}$; it also follows from \cite{T03} that for $\mathcal{L}^1$ a.e.~$t \geq 0$ the varifold $c_0^{-1}V_{t,1} \in \mathbf{V}_{n-1}(\Omega)$ is integer rectifiable. We also ensure that, for the same $\mathcal{L}^1$ a.e.~$t \geq 0$ as above, we have unique boundary limiting varifolds, $V_{t,2} \in \mathbf{V}_{n-1}(\overline{\Omega})$, by setting $V_{t,2} = \sigma(1) \mathcal{H}^{n-1}\lfloor_{\{Tu = +1\}}$; the dominated convergence theorem shows that $V_{t,2}^{\varepsilon_{i_j}} \rightharpoonup V_{t,2}$ for each subsequence, $\{\varepsilon_{i_j}\} \subset \{\varepsilon_i\}$, such that $u_{j}(\,\cdot\,,t) \lfloor_{\partial\Omega} \rightarrow Tu(\,\cdot\,,t)$ in $L^1(\partial\Omega)$. Thus, for the same $\mathcal{L}^1$ a.e.~$t \geq 0$ as above, we have uniquely defined limiting varifolds, $V_t \in \mathbf{V}_{n-1}(\overline{\Omega})$, with $V_t^{\varepsilon_i} \rightharpoonup V_t$ by setting $V_t = V_{t,1} + V_{t,2}$. In general, without the rectifiability of the limit measure, $\mu_t$, for $\mathcal{L}^1$ a.e.~$t \geq 0$ (which followed from \hyperref[asmp: measure]{(A6)}) it does not appear to be enough to guarantee the uniqueness of the limiting varifolds, $V_{t,1}$, $V_{t,2}$ or $V_t$; see Appendix \ref{sec: limiting measures} for further discussion on this point. We note that the uniqueness of the limiting varifolds referred to above is implicitly understood to be with respect to dependence on the subsequence, $\{\varepsilon_i\} \subset (0,1)$, and not on a specific choice of $t \geq 0$ (for instance, as opposed to the subsequences chosen for Theorem \ref{thm: 1} part 3 and Theorems \ref{thm: 2} and \ref{thm: 3}).
\end{rem}

Our results indicate that solutions of (\ref{eqn: parabolic Allen-Cahn with contact energy introduction}) should converge, as $\varepsilon \rightarrow 0$, to an appropriate weak notion of mean curvature flow with fixed contact angle. Motivated by this, we also derive the appropriate Ilmanen type monotonicity formula for the energy measures associated to solutions of (\ref{eqn: parabolic Allen-Cahn with contact energy introduction}) in Appendix \ref{sec: monotonicity}, which should be useful in the future study of the limiting behaviour of the gradient flow of (\ref{eqn: energy}). Under appropriate control of the discrepancy term (e.g.~under the non-concentration assumptions, \hyperref[asmp: discrepancy]{(A5)} and \hyperref[asmp: measure]{(A6)}) that appears, one would hope to utilise this monotonicity formula in order to obtain (in the vein of \cite{MT15}) refined convergence, as $\varepsilon \rightarrow 0$, of the solutions to the support of the limiting varifolds.

\bigskip

The rest of the paper is structured as follows. In Section \ref{sec: general behaviour} we prove Theorem \ref{thm: 1}. In Section \ref{sec: assumption behaviour} we establish, under the assumptions of discrepancy and measure non-concentration, that Theorem \ref{thm: 2} and Theorem \ref{thm: 3} hold respectively. In Appendix \ref{sec: limiting measures} we include an approach towards establishing uniqueness of the limiting Radon measures. We conclude the paper with Appendix \ref{sec: monotonicity} in which we establish an Ilmanen type monotonicity formula, for initial data with bounded energy, valid for the energy measures associated to solutions of (\ref{eqn: parabolic Allen-Cahn with contact energy introduction}) up to the boundary under the assumption that the boundary energy is non-negative.

\section{Boundary behaviour in general}\label{sec: general behaviour}

In this section we establish Theorem \ref{thm: 1}. Unless stated otherwise, we assume that assumptions \hyperref[asmp: on W]{(A1)}-\hyperref[asmp: uniform bounds]{(A4)} hold throughout this section. 

\bigskip

We first establish the following semi-decreasing property for the energy measures associated to solutions of (\ref{eqn: parabolic Allen-Cahn with contact energy introduction}):

\begin{lem}\label{lem: semidecreasing property}
Suppose for some $\varepsilon \in (0,1)$ that $u \in C^\infty(\overline{\Omega} \times [0,\infty))$ is a solution to (\ref{eqn: parabolic Allen-Cahn with contact energy introduction}) with $E_\varepsilon(u(\, \cdot \,, 0)) \leq E_0$. Then, for each $\phi \in C^2(\overline{\Omega};\mathbb{R}^+)$, the function
$$t \mapsto \int_{\overline{\Omega}} \phi \, d\mu_t^\varepsilon - E_0 ||\phi||_{C^2(\overline{\Omega})}t$$
is a monotone decreasing for $t \geq 0$.
\end{lem}

\begin{proof}
    Integrating by parts and using the boundary condition we have
    \begin{align*}
        \frac{d}{dt}\int_{\overline{\Omega}} \phi \, d\mu_t^\varepsilon &= \int_\Omega \phi \left(\varepsilon (\nabla u \cdot \nabla u_t) + \frac{W'(u)}{\varepsilon}u_t\right) + \int_{\partial\Omega} \phi\sigma'(u)u_t\,d\mathcal H^{n-1}\\
        &= -\int_\Omega \varepsilon (\nabla \phi \cdot \nabla u) u_t - \int_\Omega \varepsilon \phi (u_t)^2\\
        &= \int_\Omega \frac{\varepsilon (\nabla \phi \cdot \nabla u)^2}{4\phi} - \int_\Omega \varepsilon\phi \left(u_t + \frac{\nabla \phi \cdot \nabla u}{2\phi}\right)^2\\
        &\leq \int_\Omega \frac{\varepsilon (\nabla \phi \cdot \nabla u)^2}{4\phi}\\
        &\leq \int_\Omega \frac{|\nabla \phi|^2}{2\phi}\left(\frac{\varepsilon|\nabla u|^2}{2} + \frac{W(u)}{\varepsilon}\right)
        \leq E_0 ||\phi||_{C^2(\overline{\Omega})},
    \end{align*}
    where in the third equality we complete the square and in the final inequality use the estimate $\sup_{\{\phi > 0\}} \frac{|\nabla \phi|^2}{\phi} 
    \leq 2 \sup ||\nabla^2 \phi||$ (which follows from Cauchy's mean value theorem). The desired monotonicity follows by integrating in $t$.
\end{proof}

We now establish Theorem \ref{thm: 1} part 1:
\begin{proof}[Proof of Theorem \ref{thm: 1} part 1]
With Lemma \ref{lem: semidecreasing property} established, this now follows in an identical manner to \cite[Proposition 5.2]{MT15}; we repeat the argument here for completeness.

\bigskip

Let $B_0 \subset [0,\infty)$ be a countable dense subset. By the energy bound in \hyperref[asmp: uniform bounds]{(A4)} and the compactness of Radon measures there exists a subsequence (denoted by the same index) and Radon measures, $\{\mu_t\}_{t \in B_0}$, supported in $\overline{\Omega}$ such that, for $t \in B_0$, we have $\mu_t^{\varepsilon_i} \rightharpoonup \mu_t$.

\bigskip

Let $\{\phi_k\}_{k \geq 1} \subset C^2(\overline{\Omega}; \mathbb{R}_+)$ be a countable dense subset of $C(\overline{\Omega}; \mathbb{R}_+)$. By Lemma \ref{lem: semidecreasing property}, for each $k \geq 1$ there is an at most countable subset, $B_k \subset [0, \infty)$, such that the function $t \in B_0 \mapsto \mu_t(\phi_k)$ can be extended continuously to $[0, \infty) \setminus B_k$. Moreover, by setting $B = \cup_{k \geq 1} B_k$ (which is countable) we ensure that, for each $k \geq 1$, the function $t \mapsto \mu_t(\phi_k)$ defined on $B_0$ has the continuous extension to $[0, \infty) \setminus B$ given by
\begin{equation}\label{eqn: continuous extension}
    \mu_s(\phi_k) = \lim_{\substack{t \nearrow s \\ t \in B_0}} \mu_t(\phi_k) = \lim_{\substack{t \searrow s \\ t \in B_0}} \mu_t(\phi_k)
\end{equation}
for each $s \in [0, \infty) \setminus B$. Let $s \in [0, \infty) \setminus B$ and $\{\varepsilon_{i_j}\} \subset \{\varepsilon_i\}$ be any subsequence such that
\begin{equation}\label{eqn: convergence for good times}
    \mu_s^{j} \rightharpoonup \tilde{\mu}_s
\end{equation}
for some Radon measure, $\tilde{\mu}_s$. For each $t,\tilde{t} \in B_0$ with $t < s < \tilde{t}$ and $k \geq 1$ we have from Lemma \ref{lem: semidecreasing property} that
\begin{equation*}
    \mu_{\tilde{t}}^j(\phi_k) - E_0||\phi_k||_{C^2(\overline{\Omega})}(\tilde{t} - s) \leq \mu_s^j(\phi_k) \leq \mu_{t}^j(\phi_k) - E_0||\phi_k||_{C^2(\overline{\Omega})}(t - s),
\end{equation*}
and thus by (\ref{eqn: continuous extension}) and (\ref{eqn: convergence for good times}) we see that
\begin{equation*}
     \mu_{\tilde{t}}(\phi_k) - E_0||\phi_k||_{C^2(\overline{\Omega})}(\tilde{t} - s) \leq \tilde{\mu}_s(\phi_k) \leq \mu_{t}(\phi_k) - E_0||\phi_k||_{C^2(\overline{\Omega})}(t - s).
\end{equation*}
Upon sending $t \nearrow s$ and $\tilde{t} \searrow s$ we have $\mu_s(\phi_k) = \tilde{\mu}_s(\phi_k)$ and so $\mu_s^{\varepsilon_i}(\phi_k) \rightarrow \mu_s(\phi_k)$ for all $s \in [0, \infty) \setminus B$; since the $\{\phi_k\}_{k \geq 1}$ are dense in $C(\overline{\Omega}; \mathbb{R}_+)$ we conclude that $\mu_s \rightharpoonup \mu_s$ for all $s \in [0, \infty) \setminus B$. Finally, as $B$ is countable we use the energy bound in \hyperref[asmp: uniform bounds]{(A4)}, the compactness of Radon measures, and the diagonal argument to choose a further subsequence (again denoted by the same index) such that, for each $t \geq 0$, we have $\mu_t^{\varepsilon_i} \rightarrow \mu_t$ for some collection of Radon measures, $\{\mu_t\}_{t \geq 0}$, as desired.
\end{proof}

We now establish Theorem \ref{thm: 1} part 2:
\begin{proof}[Proof of Theorem \ref{thm: 1} part 2] We denote 
\begin{equation}\label{eqn: definition of Phi}
    \Phi (t) = \int_{-1}^{t} \sqrt{2W (s)} \, ds
\end{equation}
so that $\Phi(1) = c_0$, and set $w_{i} ( \, \cdot \, , t) = \Phi ( u_{i} ( \, \cdot \, , t))$ for each $t \in [0, T)$. By (\ref{eqn: time derivative of allen cahn energy}) and the energy bound in \hyperref[asmp: uniform bounds]{(A4)} we have that
\begin{eqnarray*}
    \int_{0}^{T} \int_{\Omega} |\partial_{t} w_{i}| + |\nabla w_{i}| &=& \int_{0}^{T} \int_{\Omega} |\sqrt{2W (u_{i})}| |\partial_{t} u_{i}| + |\sqrt{2W (u_{i})}| |\nabla u_{i}| \\
    &\leq& \frac{1}{2} \int_{0}^{T} \int_{\Omega} \varepsilon_{i} (\partial_{t} u_{i})^{2} + \int_{0}^{T} \int_{\Omega} \frac{\varepsilon_{i} |\nabla u_{i}|^{2}}{2} + \frac{W (u_{i})}{\varepsilon_{i}} \\
    &\leq& \frac{1}{2} (E_{0} - E_{\varepsilon_{i}} (u_{i} ( \, \cdot \, , T))) + \int_{0}^{T} E_{\varepsilon_{i}} (u_{i} ( \, \cdot \, , t)) \, dt - \int_{0}^{T} \int_{\partial \Omega} \sigma (u_{i} ( \, \cdot \, , t)) \, d \mathcal{H}^{n - 1} \, dt.
\end{eqnarray*}
Thus we deduce that
\begin{equation*}
    \sup_{i} || w_{i} ||_{BV (\Omega \times [0, T))} < + \infty,
\end{equation*}
so that, after taking a subsequence and reindexing, there exists $w \in BV (\Omega \times [0, T))$ with $w_{i} \rightarrow w$ in $L^{1} (\Omega \times [0, T))$ and 
\begin{equation*}
    \int_{0}^{T} \int_{\Omega} |\partial_{t} w| + |\nabla w| \, dx \, dt \leq \liminf_{i \rightarrow \infty} \int_{0}^{T} \int_{\Omega} |\partial_{t} w_{i}| + |\nabla w_{i}|.
\end{equation*}
We further deduce that, after potentially taking a further subsequence and reindexing, for $\mathcal{L}^1$ a.e.~$t \in [0, T)$ we have $w_{i} (\, \cdot \, , t) \rightarrow w ( \, \cdot \, , t)$ in $L^{1} (\Omega)$.
Moreover, for such $t \in [0, T)$, $X \in C^{1}_{c} (\Omega; \mathbb{R}^{n})$, and any $\tau > 0$, we then choose large enough $i$, such that
\begin{eqnarray*}
    \int_{\Omega} w ( \, \cdot \, , t) \, \mathrm{div} X \, dx &\leq& \int_{\Omega} w_{i} ( \, \cdot \, , t) \, \mathrm{div} X \, dx + \tau \\
    &\leq& ||X||_{L^{\infty}} \int_{\Omega} \frac{\varepsilon_{i} |\nabla u_{i} ( \, \cdot \, , t)|^{2}}{2} + \frac{W (u_{i} ( \, \cdot \, , t))}{\varepsilon_{i}} \, dx + \tau.
\end{eqnarray*}
This implies that for $\mathcal{L}^1$ a.e.~$t \in [0, T)$ we have $w ( \, \cdot \, , t) \in BV (\Omega)$ with 
\begin{equation*}
    \int_{\Omega} |D w ( \, \cdot \, , t)| \leq \liminf_{i \rightarrow \infty} \int_{\Omega} \frac{\varepsilon_{i} |\nabla u_{i} ( \, \cdot \, , t)|^{2}}{2} + \frac{W (u_{i} ( \, \cdot \, , t))}{\varepsilon_{i}} \, dx.
\end{equation*}
As $\Phi$ is increasing and continuous it has a continuous inverse, $\Phi^{-1}$, and thus the functions $u_i(\,\cdot\,,t) = \Phi^{-1}(w_i(\,\cdot\,,t))$ converge, for $\mathcal{L}^1$ a.e.~$t \geq 0$, pointwise $\mathcal{H}^n$ a.e.~to $u(\,\cdot\,,t) = \Phi^{-1}(w(\,\cdot\,,t)) \in BV(\Omega)$. Noting that $\int_\Omega W(u_i) \leq E_0 \varepsilon_{i}$ and applying Fatou's lemma we conclude that
$$0 \leq \int_\Omega W(u) \leq \liminf_{j \rightarrow \infty} \int_\Omega W(u_i) = 0.$$
Thus, by the assumptions on $W$, we have $u(\,\cdot\,,t) = \pm 1$ $\mathcal{L}^{n}$ a.e.~on $\Omega$, as desired.
\end{proof}

In order to control the energy on the boundary, we first establish the following expression for the first variation of the diffuse varifolds:

\begin{prop}\label{prop: first variation of varifolds}
Suppose for some $\varepsilon \in (0,1)$ that $u \in C^\infty(\overline{\Omega} \times [0,\infty))$ is a solution to (\ref{eqn: parabolic Allen-Cahn with contact energy introduction}). Then, for $g \in C^{1}_{c} (\mathbb{R}^n; \mathbb{R}^{n})$, we have
\begin{equation}\label{eqn: first variation of parabolic allen cahn varifolds}        \begin{split}
        \delta V_{t}^{\varepsilon} (g) =& \int_{\Omega \cap \{ |\nabla u| \not= 0\}} \left( \nabla g \cdot \frac{\nabla u}{|\nabla u|} \otimes \frac{\nabla u}{|\nabla u|} \right) \, d \xi_{t}^{\varepsilon}\\
        & \quad + \int_{\Omega \cap \{\nabla u = 0\}} (\nabla g \cdot I) \, d \xi_{t}^{\varepsilon} + \int_{\Omega} \varepsilon \partial_{t} u \, (g \cdot \nabla u) \, dx \\ 
        & \quad\quad\quad +\int_{\partial \Omega} \left( \frac{\varepsilon |\nabla u|^{2}}{2} + \frac{W (u)}{\varepsilon} - \frac{\sigma' (u)^{2}}{\varepsilon} \right) \, (g \cdot \nu) \, d \mathcal{H}^{n - 1}.
    \end{split}
\end{equation}
\end{prop}

\begin{proof}
    By an identical computation to that of \cite[Lemma 4.1]{KT18}, for $g \in C^{1}_{c} (\mathbb{R}^n; \mathbb{R}^{n})$ we have
    \begin{equation*}
    \begin{split}
        \delta V_{t, 1}^{\varepsilon} (g) = & \int_{\Omega \cap \{ |\nabla u| \not= 0\}} \nabla g \cdot \frac{\nabla u}{|\nabla u|} \otimes \frac{\nabla u}{|\nabla u|} \left( \frac{\varepsilon |\nabla u|^{2}}{2} - \frac{W (u)}{\varepsilon} \right) \, dx \\
        & \quad - \int_{\Omega \cap \{\nabla u = 0\}} \nabla g \cdot I \, \frac{W (u)}{\varepsilon} \, dx + \int_{\Omega} \varepsilon \partial_{t} u \, (g\cdot \nabla u) \, dx \\ 
        &\quad\quad\quad + \int_{\partial \Omega} \left( \frac{\varepsilon |\nabla u|^{2}}{2} + \frac{W (u)}{\varepsilon} \right) \, (g \cdot \nu) \, d \mathcal{H}^{n - 1} + \int_{\partial \Omega} \sigma' (u) \, (g\cdot \nabla u) \, d \mathcal{H}^{n - 1}.
    \end{split}
\end{equation*}
    For $\delta V_{t, 2}^{\varepsilon} (g)$ we compute that for $g \in C^{1}_{c} (\mathbb{R}^n; \mathbb{R}^{n})$ by the divergence theorem we have
    \begin{equation*}
        \delta V_{t, 2}^{\varepsilon} (g) = \int_{\partial \Omega} \sigma (u) \,\mathrm{div}_{\partial\Omega} (g)  \, d \mathcal{H}^{n - 1}
        = - \int_{\partial \Omega} \sigma' (u) \, (g \cdot \nabla^\top u) \, d \mathcal{H}^{n - 1}.
    \end{equation*}
Summing the above two expressions and using the boundary condition in (\ref{eqn: parabolic Allen-Cahn with contact energy}) we see that
\begin{equation*}
    \begin{split}
        \delta V_{t}^{\varepsilon} (g)
        =& \int_{\Omega \cap \{ |\nabla u| \not= 0\}} \nabla g \cdot \frac{\nabla u}{|\nabla u|} \otimes \frac{\nabla u}{|\nabla u|} \left( \frac{\varepsilon |\nabla u|^{2}}{2} - \frac{W (u)}{\varepsilon} \right) \, dx \\
        & \quad- \int_{\Omega \cap \{\nabla u = 0\}} \nabla g \cdot I \, \frac{W (u)}{\varepsilon} \, dx + \int_{\Omega} \varepsilon \partial_{t} u \, (g \cdot \nabla u) \, dx \\ 
        &\quad\quad\quad + \int_{\partial \Omega} \left( \frac{\varepsilon |\nabla u|^{2}}{2} + \frac{W (u)}{\varepsilon} - \frac{\sigma' (u)^{2}}{\varepsilon} \right) \, (g \cdot \nu) \, d \mathcal{H}^{n - 1},
    \end{split}
\end{equation*}
which is equivalent to (\ref{eqn: first variation of parabolic allen cahn varifolds}), as desired.
\end{proof}

We also establish the following boundary energy control: 

\begin{prop}\label{prop: boundary energy bound}
    Suppose for some $\varepsilon \in (0,1)$ that $u \in C^\infty(\overline{\Omega} \times [0,\infty))$ is a solution to (\ref{eqn: parabolic Allen-Cahn with contact energy introduction}) with $E_\varepsilon(u(\, \cdot \,, 0)) \leq E_0$ and $||u(\,\cdot\,,0)||_{L^\infty(\overline{\Omega})} \leq 1$. Then, there exists a constant $C > 0$ depending only on $\Omega, E_{0}, \sigma$ and $c_{1}$ such that
    \begin{equation*}
        \int_{0}^{T} \int_{\partial \Omega} \left( \frac{\varepsilon |\nabla u (\, \cdot \, , t)|^{2}}{2} + \frac{W (u ( \, \cdot \, , t))}{\varepsilon} \right) \, d \mathcal{H}^{n - 1} \, dt \leq C (1 + T).
    \end{equation*}
    \end{prop}

\begin{proof}
    By the computation in (\ref{eqn: time derivative of allen cahn energy}), the assumptions that both $E_\varepsilon(u(\, \cdot \,, 0)) \leq E_0$ and $||u(\,\cdot\,,0)||_{L^\infty(\overline{\Omega})} \leq 1$ we have
    \begin{equation*}
        \int_{\Omega} \left( \frac{\varepsilon |\nabla u (\, \cdot \,, t)|^{2}}{2} + \frac{W (u ( \, \cdot \, , t))}{\varepsilon} \right) \leq E_{0} - \int_{\partial \Omega} \sigma (u) \, d \mathcal{H}^{n - 1} \leq E_{0} -  \left(\min_{s \in [-1,1]} \sigma(s)\right) \mathcal{H}^{n - 1} (\partial \Omega),
    \end{equation*}
    for each $t \in [0, T)$. Using the gradient of the signed distance function to $\partial \Omega$ and cutting off in an appropriately small tubular neighbourhood of $\partial \Omega$, we consider a function, $g \in C^{2}_{c} (\mathbb{R}^{n}; \mathbb{R}^{n})$, such that
    \begin{equation*}
        \begin{cases}
            || g ||_{L^{\infty}} = 1, \\
            || \nabla g ||_{L^{\infty}} \leq C (\Omega), \\
            g = \nu & \text{on} \, \partial \Omega.
        \end{cases}
    \end{equation*}
    By definition of the first variation we have $|\delta V^{\varepsilon}_{t} (g)| \leq C$, where the constant depends only on $\Omega$, $E_{0}$ and $\sigma$, and thus by (\ref{eqn: first variation of parabolic allen cahn varifolds}) we see that 
    \begin{equation*}
            \int_{\partial \Omega} \left( \frac{\varepsilon |\nabla u (\,\cdot\,,t)|^{2}}{2} + \frac{W (u ( \, \cdot \, , t))}{\varepsilon} \right) \leq - \int_{\Omega} \varepsilon (\partial_{t} u) \, (g \cdot \nabla u) \, dx + \int_{\partial \Omega} \frac{\sigma' (u)^{2}}{\varepsilon} \, d \mathcal{H}^{n - 1} + C.
    \end{equation*}
    We have 
    \begin{equation*}
        \left| \int_{\Omega} \varepsilon (\partial_{t} u) \, (g \cdot \nabla u) \, dx \right| \leq \frac{1}{2} \int_{\Omega} \varepsilon (\partial_{t} u)^{2} \, d x + \int_{\Omega} \frac{\varepsilon |\nabla u|^{2}}{2} \, dx \leq \frac{1}{2} \int_{\Omega} \varepsilon (\partial_{t} u)^{2} \, dx + C,
    \end{equation*}
    and by \hyperref[asmp: on sigma]{(A2)} there is $c_1 \in [0, 1)$ such that
    \begin{equation*}
        \int_{\partial \Omega} \frac{\sigma' (u)^{2}}{\varepsilon} \, d \mathcal{H}^{n - 1} \leq c_1^2 \int_{\partial \Omega} \frac{W(u)}{\varepsilon} \, d \mathcal{H}^{n - 1} \leq c_1 \int_{\partial \Omega} \frac{\varepsilon |\nabla u|^{2}}{2} + \frac{W(u)}{\varepsilon} \, d \mathcal{H}^{n - 1}.
    \end{equation*}
    Combining the above, we see that by (\ref{eqn: time derivative of allen cahn energy}) we have
    \begin{equation*}
        (1 - c_1^2) \int_{\partial \Omega} \left( \frac{\varepsilon |\nabla u (\,\cdot\,, t)|^{2}}{2} + \frac{W (u ( \, \cdot \, , t))}{\varepsilon} \right) \leq \frac{1}{2} \int_{\Omega} \varepsilon (\partial_{t} u)^{2} \, dx + C = - \frac{1}{2} \frac{d}{dt} E_{\varepsilon} (u) + C.
    \end{equation*}
    Integrating this expression over $[0, T)$ we have 
    \begin{equation*}
        \int_{0}^{T} \int_{\partial \Omega} \left( \frac{\varepsilon |\nabla u (\, \cdot \, , t)|^{2}}{2} + \frac{W (u ( \, \cdot \, , t))}{\varepsilon} \right) \, d\mathcal{H}^{n - 1} \, dt \leq \frac{1}{2 (1 - c_{1}^2)} (E_{\varepsilon} (u ( \, \cdot \, , 0)) - E_{\varepsilon} (u ( \, \cdot \, , T))) + C T, 
    \end{equation*}
    where $C$ now depends only on $\Omega, E_{0}, \sigma$ and $c_{1}$.
    Noting that
    \begin{equation*}
        E_{\varepsilon} (u( \, \cdot \,, T)) \geq \int_{\partial \Omega} \sigma (u (\, \cdot \, , t)) \, d \mathcal{H}^{n - 1} \geq \left(\min_{s \in [-1,1]} \sigma(s)\right)  \mathcal{H}^{n - 1} (\partial \Omega),
    \end{equation*}
    we have
    \begin{equation*}
        \int_{0}^{T} \int_{\partial \Omega} \left( \frac{\varepsilon |\nabla u (\, \cdot \, , t)|^{2}}{2} + \frac{W (u ( \, \cdot \, , t))}{\varepsilon} \right) \, d\mathcal{H}^{n - 1} \, dt \leq C (1 + T), 
    \end{equation*} 
    for a new constant, $C$, also depending only on $\Omega, E_{0}, \sigma$ and $c_{1}$.
\end{proof}

We can now complete the proof of Theorem \ref{thm: 1} part 3: 

\begin{proof}[Proof of Theorem \ref{thm: 1} part 3]
    By Fatou's Lemma and Proposition \ref{prop: boundary energy bound} we have that \begin{eqnarray*}
    && \int_{0}^{T} \liminf_{i \rightarrow \infty} \int_{\partial \Omega} \left( \frac{\varepsilon_{i} |\nabla u_{i} (\, \cdot \, , t)|^{2}}{2} + \frac{W (u_{i} ( \, \cdot \, , t))}{\varepsilon_{i}} \right) \, d\mathcal{H}^{n - 1} \, dt \\
    &&\quad\quad\quad\quad\leq \liminf_{i \rightarrow \infty} \int_{0}^{T} \int_{\partial \Omega} \left( \frac{\varepsilon_{i} |\nabla u_{i} (\, \cdot \, , t)|^{2}}{2} + \frac{W (u_{i} ( \, \cdot \, , t))}{\varepsilon_{i}} \right) \, d\mathcal{H}^{n - 1} \, dt \leq C (1 + T).
    \end{eqnarray*}
    Thus, for $\mathcal{L}^1$ a.e.~$t \in [0, T)$ we have a subsequence, $\{\varepsilon_{i_{j}} \} \subset \{ \varepsilon_{i} \}$, dependent on $t \in [0,T)$, such that 
    \begin{equation*}
        \sup_{j} \, \int_{\partial \Omega} \left( \frac{\varepsilon_{i_j} |\nabla u_j (\, \cdot \, , t)|^{2}}{2} + \frac{W (u_j ( \, \cdot \, , t))}{\varepsilon_{i_{j}}} \right) \, d\mathcal{H}^{n - 1} < + \infty.
    \end{equation*}
    By standard arguments (e.g.~see \cite[page 52]{HT00}), there exists a function $\tilde{u} ( \, \cdot \, , t) \in BV (\partial \Omega)$ such that $\tilde{u} ( \, \cdot \, , t) = \pm 1$ $\mathcal{H}^{n-1}$ a.e., $u_j ( \, \cdot \, , t) \rightarrow \tilde{u} ( \, \cdot \, , t)$ in $L^{1} (\partial \Omega)$, and 
    \begin{equation*}
        \int_{\partial \Omega} |D \tilde{u} ( \, \cdot \, , t)| \leq \liminf_{j \rightarrow \infty} \int_{\partial \Omega} |D u_j ( \, \cdot \, , t)|.
    \end{equation*}
    Moreover, we have that $\{\tilde{u} ( \, \cdot \, , t) = + 1\}$ is a Caccioppoli set in $\partial \Omega$.
\end{proof}

\section{Boundary behaviour with assumptions}\label{sec: assumption behaviour}

In this section, under the boundary non-concentration assumptions on the discrepancy and the interior measures, we establish Theorems \ref{thm: 2} and \ref{thm: 3}, respectively.

\subsection{Limiting behaviour under discrepancy non-concentration}

Unless stated otherwise, we assume that assumptions \hyperref[asmp: on W]{(A1)}-\hyperref[asmp: discrepancy]{(A5)} hold throughout this subsection.

\begin{proof}[Proof of Theorem \ref{thm: 2} part 1]

The energy bound in \hyperref[asmp: uniform bounds]{(A4)} combined with (\ref{eqn: time derivative of allen cahn energy}) and the proof of Theorem \ref{thm: 1} part 3 show that for $\mathcal{L}^1$ a.e.~$t \geq 0$ we have
\begin{equation}\label{eqn: liminf energy control}
        \liminf_{i \rightarrow \infty}\left\{ \left(\int_\Omega \varepsilon_{i}(\partial_t u)^2 \right)^\frac{1}{2} + \int_{\partial\Omega} \frac{\varepsilon_{i}|\nabla u|^2}{2} + \frac{W(u)}{\varepsilon_i} \, d \mathcal{H}^{n - 1} \right\} = c(t) < \infty.
    \end{equation}
We now fix one of the $\mathcal{L}^1$ a.e.~$t \geq 0$ such that (\ref{eqn: liminf energy control}) and \hyperref[asmp: discrepancy]{(A5)} hold, and then choose a subsequence, $\{\varepsilon_{i_j}\} \subset \{\varepsilon_i\}$, so that both $V_{t,1}^{j} \rightharpoonup V_{t,1}$ and $V_{t}^{j} \rightharpoonup V_{t}$; in particular, this ensures that $\delta V_t(g) = \lim_{j \rightarrow \infty}\delta V_t^j(g)$ for each $g \in C^1(\overline{\Omega};\mathbb{R}^n)$. By Proposition \ref{prop: boundary energy bound} we have, potentially taking a further subsequence, that $u_{j} ( \, \cdot \, , t) \rightarrow \tilde{u} (\, \cdot \, , t)$ in $L^{1} (\partial \Omega)$ as in Theorem \ref{thm: 1} part 3.

\bigskip 

For a subsequence as above, we define linear functionals, $L_j$, for $g \in C^{1}_{c} (\mathbb{R}^{n}; \mathbb{R}^{n})$ by setting
\begin{equation*}
    L_{j} (g) = \int_{\Omega} \varepsilon_{i_j} \, \partial_{t} u_{j} \, (g \cdot \nabla u_{j}) \, dx.
\end{equation*}
Noting that we then have
\begin{equation}\label{eqn: bound on L_i}
    |L_{j} (g)| \leq c(t) \left( ||V_{t, 1}^j|| (\text{spt} (g)) \right)^{\frac{1}{2}} ||g||_{L^{\infty}},
\end{equation}
after potentially taking a further subsequence,  there exists a bounded linear functional, $L$, defined on $C_{c} (\mathbb{R}^{n}; \mathbb{R}^{n})$, with $L_{j} \rightharpoonup L$; notice that by (\ref{eqn: first variation of parabolic allen cahn varifolds}), for each $g \in C_{c}^{1} (\mathbb{R}^{n}; \mathbb{R}^{n})$ with $(g \cdot \nu) = 0$ we then have $L(g) = \delta V_t(g)$. By defining 
\begin{equation*}
    || L || (U) = \sup \{ L (g) \colon g \in C_{c} (U; \mathbb{R}^{n}), \, |g| \leq 1 \},
\end{equation*}
on open sets $U \subset \mathbb{R}^{n}$ ($||L||$ will then be a Radon measure by the Riesz representation theorem, e.g.~see \cite[Theorem 1.38]{EG15}), we have by (\ref{eqn: bound on L_i}) that $||L|| \ll ||V_{t, 1}||$; therefore by the Radon--Nikodym theorem there exists a $||V_{t, 1}||$ measurable vector field, $H_t$, such that $L(g) = -\int_{\overline{\Omega}} g \cdot H_t \, d ||V_{t,1}||$.
Moreover, we note that, for $g \in C^{1}_{c} (\mathbb{R}^{n}; \mathbb{R}^{n})$, we have by similar reasoning leading to (\ref{eqn: bound on L_i}) that
\begin{equation}\label{eqn: L2 bounds on mean curvature}
    \int_{\overline{\Omega}} g \cdot H_t \, d ||V_{t,1}|| = - \lim_{j \rightarrow \infty} L_{j} (g) \leq c (t) \limsup_{j \rightarrow \infty} \left( \int_{\overline{\Omega}} |g|^{2} \, d \| V_{t, 1}^{j} \| \right)^{\frac{1}{2}} = c (t) \left( \int_{\overline{\Omega}} |g|^{2} \, d \| V_{t, 1} \| \right)^{\frac{1}{2}},
\end{equation}
which implies that $H_{t} \in L^{2}_{\| V_{t, 1} \|} (\overline{\Omega}; \mathbb{R}^{n})$ by \cite[Lemma A.3]{L17}. Recalling (\ref{eqn: first variation of parabolic allen cahn varifolds}), assumption \hyperref[asmp: discrepancy]{(A5)}, and the reasoning above, for each $g \in C_{c}^{1} (\mathbb{R}^{n}; \mathbb{R}^{n})$ with $(g \cdot \nu) = 0$ we have 
\begin{equation*}
    \delta V_{t, 1} (g) + \sigma(1) \int_{\{ \tilde{u} ( \, \cdot \, , t) = + 1\}} \mathrm{div}_{\partial \Omega} g \, d \mathcal{H}^{n - 1} = - \int_{\overline{\Omega}} g \cdot H_t \, d ||V_{t,1}||,
\end{equation*}
as desired.\end{proof}

\begin{proof}[Proof of Theorem \ref{thm: 2} part 2]

Fix one of the $\mathcal{L}^1$ a.e.~$t \geq 0$ such that (\ref{eqn: liminf energy control}) and \hyperref[asmp: discrepancy]{(A5)} hold, and then choose a subsequence, $\{\varepsilon_{i_j}\} \subset \{\varepsilon_i\}$, so that $V_{t}^{j} \rightharpoonup V_{t}$; in particular, this ensures that $\delta V_t(g) = \lim_{j \rightarrow \infty}\delta V_t^j(g)$ for each $g \in C^1(\overline{\Omega};\mathbb{R}^n)$. As $||V_t^j|| = \mu_t^j$ we then deduce, for this $t \geq 0$ above, that by Theorem \ref{thm: 1} part 1 we have $||V_t|| = \mu_t$.

\bigskip

In order to control $||\delta V_t||$ we bound each of the terms appearing in (\ref{eqn: first variation of parabolic allen cahn varifolds}). Firstly, for each $g \in C^{1}(\overline{\Omega};\mathbb{R}^n)$ we have

$$ \int_{\Omega \cap \{|\nabla u_{j}| \neq 0\}} \nabla g \cdot \frac{\nabla u_{j}}{|\nabla u_{j}|} \otimes \frac{\nabla u_{j}}{|\nabla u_{j}|} \,d\xi_t^{\varepsilon_{i_j}} - \int_{\Omega \cap \{\nabla u_{j} = 0\}}(\nabla g \cdot I) \, d\xi_t^{\varepsilon_{i_j}} \leq 2 \sup |\nabla g| \int_\Omega|\xi_t^{\varepsilon_{i_j}}|(x)\,dx,$$
which converges to zero as $j \rightarrow \infty$ by \hyperref[asmp: discrepancy]{(A5)}. Next, as in the proof of Proposition \ref{prop: boundary energy bound} we see that for sufficiently large $j$ and a constant, $C$, depending only on $n$, $\Omega$, $\sigma$, and the energy bound in \hyperref[asmp: uniform bounds]{(A4)} we have
$$ \int_\Omega \varepsilon_{i_j}\partial_t u_j (g \cdot \nabla u_j) \leq C c(t)\sup|g|,$$
and
$$\int_{\partial\Omega} \left(\frac{\varepsilon_{i_j}|\nabla u_{j}|^2}{2} + \frac{W(u_{j})}{\varepsilon_{i_j}} - \frac{\sigma'(u_{j})^2}{\varepsilon_{i_j}}\right)(g \cdot \nu) \, d\mathcal{H}^{n-1} \leq 3c(t) \sup|g|.$$
Combining the above three bounds and comparing with (\ref{eqn: first variation of parabolic allen cahn varifolds}), we conclude that there exists some constant, $C(t)$, such that 
\begin{equation}\label{eqn: first variation of V t 1 is bounded}
    |\delta V_t(g)| = \lim_{j \rightarrow \infty} |\delta V_t^j(g)| \leq C(t)\sup|g|.
\end{equation}
This shows that for $\mathcal{L}^1$ a.e.~$t \geq 0$ we have $||\delta V_t|| (\overline{\Omega}) \leq C(t) < \infty$. We conclude by integrating $||\delta V_t||(\overline{\Omega})$ between $0$ and $T$; noting that $C(t)$ is locally integrable as a function of $t$ by the choice of subsequence satisfying (\ref{eqn: liminf energy control}), the energy bound in \hyperref[asmp: uniform bounds]{(A4)}, and Proposition \ref{prop: boundary energy bound}.
\end{proof}

\subsection{Limiting behaviour under measure non-concentration}

Unless stated otherwise, we assume that assumptions \hyperref[asmp: on W]{(A1)}-\hyperref[asmp: measure]{(A6)} hold throughout this subsection.

\begin{proof}[Proof of Theorem \ref{thm: 3} part 1] 
Fix one of the $\mathcal{L}^1$ a.e.~$t \in [0, T)$ such that \hyperref[asmp: measure]{(A6)} holds and a subsequence, $\{\varepsilon_{i_{j}}\} \subset \{ \varepsilon_{i}\}$, dependent on this $t$, as in Theorem \ref{thm: 1} part 3 so that both $w_{i_j} ( \, \cdot \, , t) = \Phi (u_{j} ( \, \cdot \, , t) ) \rightarrow w ( \, \cdot \, , t) = \Phi (u ( \, \cdot \, , t))$ in  $L^{1} (\Omega)$ and $w_{i_{j}} ( \, \cdot \, , t) = \Phi (u_j ( \, \cdot \, , t)) \rightarrow \tilde{w} ( \, \cdot \, , t) = \Phi (\tilde{u} ( \, \cdot \, , t))$ in  $L^{1} (\partial \Omega)$; where $\Phi$ is defined as in (\ref{eqn: definition of Phi}).

\bigskip

As $w ( \, \cdot \, , t) \in BV (\Omega)$ we have (e.g.~see \cite[Theorem 5.6]{EG15}), for $X \in C^{1}_{c} (\mathbb{R}^{n}; \mathbb{R}^{n})$, that
\begin{equation}\label{eqn: BV trace formula}
    \int_{\partial \Omega} (X \cdot \nu) \, T w ( \, \cdot \, , t) \, d \mathcal{H}^{n - 1} = \int_{\Omega} X \cdot d [D w ( \, \cdot \, , t)] + \int_{\Omega} w ( \, \cdot \, , t) \, \mathrm{div}X \, dx,
\end{equation}
where $T$ denotes the trace operator defined for functions in $BV(\Omega)$. Fixing $f \in C^{1} (\partial \Omega)$ and $\delta \in (0, \kappa)$, by using the gradient of the signed distance function to $\partial \Omega$ and an appropriate cutoff function, one can construct a vector field, $X \in C^{1}_{c} (\mathbb{R}^{n}; \mathbb{R}^{n})$, such that $(X \cdot \nu) = f$ on $\partial \Omega$ and $\text{spt}(X) \cap \Omega \subset N_{\delta}$.
With this vector field, by adding and subtracting terms involving $w_{i_j}(\,\cdot\,,t)$ and $\tilde{w}(\,\cdot\,,t)$ into (\ref{eqn: BV trace formula}) and applying the divergence theorem we have
\begin{eqnarray*}
    && \int_{\partial \Omega} (X \cdot \nu) \, T w ( \, \cdot \, , t) \, d \mathcal{H}^{n - 1} \\
    && \hspace{1cm} = \int_{N_{\delta}} X \cdot d [D w ( \, \cdot \, , t)] + \int_{N_{\delta}} (w ( \, \cdot \, , t) - w_{i_j} ( \, \cdot \, , t)) \, \mathrm{div} \, X \, dx - \int_{N_{\delta}} \nabla w_{i_j} ( \, \cdot \, , t) \, \cdot X \, dx \\
    && \hspace{2cm} + \int_{\partial \Omega} (w_{i_j} ( \, \cdot \, , t) - \tilde{w} ( \, \cdot \, , t)) \, (X \cdot \nu) \, d \mathcal{H}^{n - 1} + \int_{\partial \Omega} \tilde{w} ( \, \cdot \, , t) \, (X \cdot \nu) \, d \mathcal{H}^{n - 1}. 
\end{eqnarray*}
By a similar calculation to the proof of Theorem \ref{thm: 1} part 2, we bound
\begin{equation*}
    \left| \int_{N_{\delta}} X \cdot d [D w ( \, \cdot \, , t)] \right| \leq || X ||_{L^{\infty}} \liminf_{j \rightarrow \infty} |D w_{i_j} (\, \cdot \, , t)| (N_{\delta}) \leq || X ||_{L^{\infty}} \liminf_{j \rightarrow \infty} ||V_{t, 1}^{j}|| (N_{\delta}),
\end{equation*}
and similarly we have the bound
\begin{equation*}
    \left| \int_{N_{\delta}} \nabla w_{i_{j}} ( \, \cdot \, , t) \, \cdot X \, dx \right| \leq || X ||_{L^{\infty}} ||V_{t, 1}^{j}|| (N_{\delta}).
\end{equation*}
Thus, upon sending $j \rightarrow \infty$ we ensure that, by the two bounds above and the convergence of $w_{i_j}(\,\cdot\,,t)$, we have
\begin{equation*}
    \int_{\partial \Omega} f \, T w ( \, \cdot \, , t) \, d \mathcal{H}^{n - 1} = \int_{\partial \Omega} f \, \tilde{w} ( \, \cdot \, , t) \, d \mathcal{H}^{n - 1}.
\end{equation*}
As $f \in C^1(\partial\Omega)$ chosen above was arbitrary, we see that $T w ( \, \cdot \, , t) = \tilde{w} ( \, \cdot \, , t) = \Phi (\tilde{u} ( \, \cdot \, , t))$ $\mathcal{H}^{n - 1}$ a.e.~on $\partial \Omega$ by the fundamental lemma of the calculus of variations. Noting that, as $u = \pm 1$ $\mathcal{L}^n$ a.e.~in $\Omega$, $\tilde{u} = \pm 1$ $\mathcal{H}^{n-1}$ a.e.~on $\partial \Omega$, and recalling that both $\Phi(1) = c_0$ and $\Phi(-1) = 0$, we have
\begin{equation*}
    \begin{cases}
        u = c_0^{-1} (2 w - c_0) & \mathcal{L}^{n} \text{ a.e.~in } \, \Omega, \\
        \tilde{u} = c_0^{-1} (2 \tilde{w} - c_0) & \mathcal{H}^{n - 1} \text{ a.e.~on } \partial \Omega.
    \end{cases}
\end{equation*}
From the above expressions, the linearity of the trace operator, $T$, and the fact that $T w ( \, \cdot \, , t) = \tilde{w} ( \, \cdot \, , t)$ as deduced above, we conclude that $Tu ( \, \cdot \, , t) = \tilde{u} ( \, \cdot \, , t)$ $\mathcal{H}^{n-1}$ a.e.~on $\partial \Omega$, as desired.
\end{proof}

\begin{proof}[Proof of Theorem \ref{thm: 3} part 2]
  Fix one of the $\mathcal{L}^1$ a.e.~$t \geq 0$ such that (\ref{eqn: liminf energy control}) and \hyperref[asmp: measure]{(A6)} hold, and a subsequence, $\{\varepsilon_{i_j}\} \subset \{\varepsilon_i\}$, so that $V_t^{j} \rightharpoonup V_t$, $V_{t,1}^{j} \rightharpoonup V_{t,1}$, and $V^{j}_{t, 2} \rightharpoonup \sigma (1) \, \mathcal{H}^{n - 1} \lfloor_{\{ T u ( \, \cdot \, , \, t) = + 1\}}$. Note that by \hyperref[asmp: measure]{(A6)} we have that $\| V_{t, 1} \| (\partial \Omega) = 0$, and so recalling Theorem \ref{thm: 2} part 1 and Remark \ref{rem: restriction of first variation}, for $g \in C^{1}_{c} (\mathbb{R}^{n}; \mathbb{R}^{n})$ with $(g \cdot \nu) = 0$, we have that 
    \begin{equation*}
        (\delta V_{t, 1}) \lfloor_{\Omega} (g) + (\delta V_{t, 1}) \lfloor_{\partial \Omega} (g) + \sigma (1) \int_{\{T u ( \, \cdot \, , \, t) = + 1\}} \mathrm{div}_{\partial \Omega} (g) \, d \mathcal{H}^{n - 1} = - \int_{\Omega} g \cdot H_{t} \, d \| V_{t, 1} \|.
    \end{equation*}
    Thus, taking $g \in C_{c}^{1} (\Omega; \mathbb{R}^{n})$, we have
    \begin{equation*}
        (\delta V_{t, 1}) \lfloor_{\Omega} (g) = - \int_{\Omega} g \cdot H_{t} \, d \| V_{t, 1} \|.
    \end{equation*}
    For $\delta \in (0, \kappa)$, consider a cutoff function, $\eta_{\delta} \in C^{\infty}_{c} (\Omega)$, which satisfies the following, 
    \begin{equation*}
        \begin{cases}
            \eta_{\delta} = 0 & \text{on } N_{\frac{\delta}{2}}, \\ 
            \eta_{\delta} = 1 & \text{on } \Omega \setminus N_{\delta}, \\
            |\eta_{\delta}| \leq 1.
        \end{cases}
    \end{equation*}
    Then, by dominated convergence theorem, for any $g \in C^{1}_{c} (\mathbb{R}^{n}; \mathbb{R}^{n})$, we have
    \begin{eqnarray*}
        (\delta V_{t, 1}) \lfloor_{\Omega} (g) = \lim_{\delta \rightarrow 0} (\delta V_{t, 1}) \lfloor_{\Omega} (\eta_{\delta} g) = - \lim_{\delta \rightarrow 0}\int_{\Omega}  (\eta_{\delta} g \cdot H_{t}) \, d \| V_{t, 1} \| = - \int_{\Omega} g \cdot H_{t} \, d \| V_{t, 1} \|.
    \end{eqnarray*}
    Thus, for $g \in C_{c}^{1} (\mathbb{R}^{n}; \mathbb{R}^{n})$ with $(g \cdot \nu) = 0$ we have by the divergence theorem that
    \begin{eqnarray*}
        (\delta V_{t, 1}) \lfloor_{\partial \Omega} (g) &=& - \sigma (1) \int_{\{T u ( \, \cdot \, , t) = + 1\}} \mathrm{div}_{\partial \Omega} (g) \, d \mathcal{H}^{n - 1} \\
        &=& - \sigma(1)\int_{\partial^*\{Tu(\, \cdot \, , \, t) = +1\}} g \cdot \vec{n}_{\{Tu(\, \cdot \, , \, t) = +1\}} \, d\mathcal{H}^{n-2},
    \end{eqnarray*}
    as desired.
\end{proof}

\appendix
\section{Limiting Radon measures}\label{sec: limiting measures}

We deduced the existence of unique (i.e.~depending on the initial subsequence, $\{\varepsilon_i\} \subset (0,1)$, and not on a choice of $t \geq 0$) limiting measures, $\{\mu_{t}\}_{t \geq 0}$, for the measures $\{\mu_{t}^{\varepsilon_i}\}_{t \geq 0}$ in Theorem \ref{thm: 1} part 1; which followed from the semi-decreasing property of Lemma \ref{lem: semidecreasing property}. One could also hope to deduce uniqueness of limiting measures for the sequences $\{\mu_{t, 1}^{\varepsilon_i}\}_{t \geq 0}$ and $\{\mu_{t, 2}^{\varepsilon_i}\}_{t \geq 0}$.
However, due to a lack of a semi-decreasing property for each of these measures, in general we can only guarantee, from the energy bound in \hyperref[asmp: gradient flow]{(A3)}, that for each $t \geq 0$ there exists a subsequence, $\{\varepsilon_{i_{j}}\} \subset \{ \varepsilon_{i}\}$, dependent on this $t \geq 0$, and limiting Radon measures, $\mu_{t, 1}$ and $\mu_{t,2}$, such that $\mu_{t, 1}^j \rightharpoonup \mu_{t, 1}$ and $\mu_{t, 2}^j \rightharpoonup \mu_{t, 2}$. Due to the dependence on a subsequence, from this we cannot necessarily ensure uniqueness of the limiting measures obtained.

\bigskip

For each $t \geq 0$ we denote by
\begin{equation}\label{eqn: Collection of limiting subsequential measures}
    \mathcal{F}_{t,j} = \{ \mu_{t, j} \colon \text{there exists a subsequence,} \, \{ \varepsilon_{i_{l}} \} \subset \{ \varepsilon_{i} \}, \text{such that } \mu_{t, j}^{l} \rightharpoonup \mu_{t, j} \},
\end{equation}
the collections of limiting Radon measures, for $j = 1,2$.

\begin{rem}
    As noted in Remark \ref{rem: canonical measures under non-concentration}, under assumption \hyperref[asmp: measure]{(A6)}, for $\mathcal{L}^1$ a.e.~$t \geq 0$ we have that $\mu_{t, 1}^{\varepsilon_i} \rightharpoonup \mu_{t} \lfloor_{\Omega}$ and $\mu_{t, 2}^{\varepsilon_i} \rightharpoonup \mu_{t} \lfloor_{\partial \Omega}$ without passing to a further subsequence, giving rise to families of unique limiting measures. Thus, for such $\mathcal{L}^1$ a.e.~$t \geq 0$ we have $\mathcal{F}_{t,1} = \{ \mu_{t} \lfloor_{\Omega}\}$ and $\mathcal{F}_{t,2} = \{\mu_{t} \lfloor_{\partial \Omega}\}$.
\end{rem}
In order to define unique limiting measures independent of a specific $t \geq 0$ under assumptions \hyperref[asmp: on W]{(A1)}-\hyperref[asmp: uniform bounds]{(A4)} (in particular without assuming \hyperref[asmp: measure]{(A6)} as in Remark \ref{rem: canonical measures under non-concentration}), we now explore a potential approach by `slicing' appropriate space-time limiting measures. We first define, for each $\varphi \in C_{c} ( \mathbb{R}^n \times [0, + \infty) )$, the Radon measures 
\begin{equation*}
    \mu^{\varepsilon_i} (\varphi) = \int_{0}^{+\infty} \mu^{\varepsilon_i}_{t} (\varphi ( \, \cdot \, , t) ) \, dt.
\end{equation*}
Then, for each $T \in (0, +\infty)$ and $\varphi \in C_{c} ( \mathbb{R}^n \times [0, T) )$, as $i \rightarrow \infty$, by the dominated convergence theorem and (\ref{eqn: uniform t and i energy bounds}) we have
\begin{equation*}
    \mu^{i} (\varphi) = \int_{0}^{T} \mu_{t}^{\varepsilon_i} (\varphi ( \, \cdot \, , t)) \, dt \rightarrow \int_{0}^{T} \mu_{t} (\varphi ( \, \cdot \, , t)) \, dt \eqqcolon \mu (\varphi).
\end{equation*}
We similarly define the measures, 
\begin{equation*}
    \mu_{j}^{\varepsilon_i} (\varphi) = \int_{0}^{+\infty} \mu_{t, j}^{\varepsilon_i} (\varphi ( \, \cdot \, , t)) \, dt,
\end{equation*}
for $j = 1,2$, so that we have Radon measures, $\mu_{1}$ and $\mu_{2}$, on $\mathbb{R}^{n} \times [0, + \infty)$, such that, after potentially taking a subsequence and reindexing, $\mu_{j}^{\varepsilon_i} \rightharpoonup \mu_{j}$.

\bigskip

Now define the projection measure of $\mu_{1}$ onto $[0, + \infty)$ by setting $\lambda_{1} (I) = \mu_{1} (\mathbb{R}^{n} \times I)$ for each $I \subset [0, + \infty)$. By \cite[Theorem 1.45]{EG15}, for $\lambda_{1}$ a.e.~$t \geq 0$, there exists a Radon measure, $\Lambda_{t, 1}$, on $\mathbb{R}^{n}$ with $\Lambda_{t, 1} (\mathbb{R}^{n}) = 1$ and such that, for each $f \in C_{c} (\mathbb{R}^{n} \times [0, + \infty))$, we have
\begin{equation*}
    \int_{\mathbb{R}^{n} \times [0, + \infty)} f \, d \mu_{1} = \int_{0}^{+ \infty} \int_{\mathbb{R}^{n}} f \, d \Lambda_{t, 1} \, d \lambda_{1}.
\end{equation*}
We then extend $\Lambda_{t, 1}$ by $0$, so to be defined for all $t \in [0, + \infty)$.
Furthermore, for $I \subset [0, + \infty)$, relatively open, recalling (\ref{eqn: uniform t and i energy bounds}), we have that
\begin{equation*}
    \lambda_{1} (I) \leq \liminf_{i \rightarrow \infty} \int_{I} \mu_{t, 1}^{\varepsilon_{i}} (\mathbb{R}^{n}) \, dt \leq \left( E_{0} + \sup_{s \in [-1, 1]} |\sigma (s)| \, \mathcal{H}^{n-1} (\partial \Omega) \right) \, \mathcal{L}^{1} (I).
\end{equation*}
Therefore $\lambda_{1} \ll \mathcal{L}^{1}$, and thus $d \lambda_{1} = F_{1} \, d \mathcal{L}^{1}$ for some $\mathcal{L}^{1}$ measurable function $F_{1} \colon [0, + \infty) \rightarrow [0, + \infty)$, with
\begin{equation*}
    \text{ess} \sup_{t \in [0, + \infty)} F_{1} (t) \leq \left( E_{0} + \sup_{s \in [-1, 1]} |\sigma (s)| \, \mathcal{H}^{n-1} (\partial \Omega) \right).
\end{equation*}
For $j = 2$, after splitting $\mu_{2}$ into its positive and negative parts $\mu_{2}^{\pm}$, and defining their respective projection measures $\lambda_{2}^{\pm}$, we may argue similarly as above, deducing the existence of Radon measures $\Lambda_{t, 2}^{\pm}$ on $\mathbb{R}^{n}$, along with $\mathcal{L}^{1}$-measureable functions $F_{2}^{\pm} \colon [0, + \infty) \rightarrow [0, + \infty)$, again with finite bounds on their essential supremums.
For $\mathcal{L}^{1}$-a.e.~$t \geq 0$ we then define our time slice Radon measures on $\mathbb{R}^{n}$, by setting
\begin{equation*}
    \tilde{\mu}_{t, 1} = F_{1} (t)  \Lambda_{t, 1}, \hspace{0.5cm}\tilde{\mu}_{t, 2} = F_{2}^{+} (t) \Lambda_{t, 2}^{+} - F_{2}^{-} (t) \Lambda_{t, 2}^{-},
\end{equation*} 
so that for each $f \in C_{c} (\mathbb{R}^{n} \times [0, + \infty))$ we have
\begin{equation*}
    \int_{\mathbb{R}^{n} \times [0, + \infty)} f \, d \mu_{j} = \int_{0}^{+ \infty} \int_{\mathbb{R}^{n}} f \, d \tilde{\mu}_{t, j} \, d t.
\end{equation*}
We relate the Radon measures $\{\mu_{t, 1}^{\varepsilon_i}\}_{t \geq 0}$ and $\{\mu_{t, 2}^{\varepsilon_i}\}_{t \geq 0}$, to our families of limiting `slicing' Radon measures, $\{\tilde{\mu}_{t, 1}^{\varepsilon_i}\}_{t \geq 0}$ and $\{\tilde{\mu}_{t, 2}^{\varepsilon_i}\}_{t \geq 0}$, by the following weak convergence. For each $\varphi \in C_{c} (\mathbb{R}^{n})$ and $\psi \in C_{c} ([0, + \infty))$, we have that
\begin{eqnarray*}
    \int_{0}^{+ \infty} \psi \int_{\mathbb{R}^{n}} \varphi \, d \mu_{t, j}^{\varepsilon_i} \, dt = \int_{\mathbb{R}^{n} \times [0, + \infty)} \psi \, \varphi \, d \mu_{j}^{\varepsilon_i} \rightarrow \int_{\mathbb{R}^{n} \times [0, + \infty)} \psi \, \varphi \, d \mu_{j} = \int_{0}^{+ \infty} \psi \int_{\mathbb{R}^{n}} \varphi \, d \tilde{\mu}_{t, j} \, dt,
\end{eqnarray*}
for $j = 1,2$. Thus, by the density of compactly supported continuous functions in $L^{1} ([0, + \infty))$, along with the uniform bounds on $\{\mu_{t}^{\varepsilon_{i}}\}_{t \in [0, + \infty)}$, and $F_{1}$, $F_{2}^{\pm}$, we deduce that for each $\varphi \in C_{c} (\mathbb{R}^{n})$, we have
\begin{equation*}
    \int_{\mathbb{R}^{n}} \varphi \, d \mu_{t, j}^{\varepsilon_i} \rightharpoonup \int_{\mathbb{R}^{n}} \varphi \, d \tilde{\mu}_{t, j},
\end{equation*}
in $(L^{1} ([0, + \infty)))^{*}$ for $j = 1,2$.
It seems unclear if one can deduce stronger convergence of the Radon measures, $\mu_{t, j}^{\varepsilon_i}$, without passing to a further subsequence dependent on $t \geq 0$, and what relations, if any, exist between the limiting `slicing' Radon measures, $ \tilde{\mu}_{t, j}$, and the collection of limiting Radon measures, $\mathcal{F}_{t,j}$, as defined in (\ref{eqn: Collection of limiting subsequential measures}).

\section{Boundary monotonicity formula}\label{sec: monotonicity}

In this appendix we derive an Ilmanen type monotonicity formula (c.f.~\cite{I93,MT15}), as speculated about in \cite[Section 5]{KT18} for the time independent case of (\ref{eqn: parabolic Allen-Cahn with contact energy introduction}), valid up to the boundary for the Radon measures, $\mu_t^\varepsilon$, associated to solutions of (\ref{eqn: parabolic Allen-Cahn with contact energy introduction}); the precise statement of this is given by Proposition \ref{prop: monotonicity formula} below.

\bigskip

Recall that for $x,y \in \mathbb{R}^n$ and $t < s$ we define the $(n-1)$-dimensional backwards heat kernel as
    $$ \rho_{(y,s)}(x,t) = \frac{1}{(4\pi (s-t))^\frac{n-1}{2}}e^{-\frac{|x-y|^2}{4(s-t)}}.$$
For each $x \in N_{\kappa}$ there exists a unique point $\xi(x) \in \partial \Omega$ realising the distance between $x$ and $\partial \Omega$ (i.e.~$\mathrm{dist}(x,\partial\Omega) = |x - \xi(x)|$); we thus define the reflection, $\tilde{x}$, of a point $x \in N_\kappa$ in $\partial\Omega$ by setting $\tilde{x} = 2\xi(x) - x$. For points $x \in N_{\kappa}$ and $y \in \mathbb{R}^n$ we then define the reflected kernel by setting $$\tilde{\rho}_{(y,s)}(x,t) = \rho_{(y,s)}(\tilde{x},t).$$
Fix a radially symmetric cutoff function, $\eta \in C^\infty_c(\mathbb{R}^n; [0,1])$, such that
$$\begin{cases}
    \eta \equiv 1 & \text{on } B_{\frac{\kappa}{4}}(0),\\
    \frac{\partial\eta}{\partial r} \leq 0 & \text{on } \mathbb{R}^n,\\
    \mathrm{spt}(\eta) \subset B_{\frac{\kappa}{2}}(0),\\
\end{cases}$$
where $r = |x|$ and then, for $x,y \in N_{\kappa}$, define the truncated kernels 
$$\begin{cases}
    \rho_1 = \rho_1(x,t) = \eta(x - y)\rho_{(y,s)}(x,t),\\
    \rho_2 = \rho_2(x,t) = \eta(\tilde{x} - y)\tilde{\rho}_{(y,s)}(x,t).
\end{cases}$$
Noting that for $x \in N_{\kappa} \setminus N_{\frac{\kappa}{2}}$ and $y \in N_\frac{\kappa}{2}$ we have $|\tilde{x} - y| > \frac{\kappa}{2}$ (so that $\eta(\tilde{x}-y) = 0$), we smoothly extend our definition of $\rho_2$ to be identically zero for points $x \in \Omega \setminus N_{\kappa}$ and $y \in N_\frac{\kappa}{2}$.

\begin{prop}\label{prop: monotonicity formula} Suppose for some $\varepsilon \in (0,1)$ that $u \in C^\infty(\overline{\Omega} \times [0,\infty))$ is a solution to (\ref{eqn: parabolic Allen-Cahn with contact energy introduction}) with $E_\varepsilon(u(\, \cdot \,, 0)) \leq E_0$ where $\sigma \geq 0$. Then, there exist constants $C_1,C_2 > 0$ depending only on $n$, $\Omega$ and $E_0$ such that
\begin{equation*}
    \frac{d}{dt}(e^{C_1(s-t)^\frac{1}{4}}\mu_t^\varepsilon(\rho_1 + \rho_2)) \leq e^{C_1(s-t)^\frac{1}{4}}\left(\int_\Omega \frac{\rho_1 + \rho_2}{2(s-t)} \, d\xi_t^\varepsilon + C_2\right).
\end{equation*}
for $s > t > 0$ and $y \in N_\frac{\kappa}{2}$. For $s > t > 0$ and $y \in \Omega \setminus N_\frac{\kappa}{2}$ we have
\begin{equation*}
    \frac{d}{dt}\mu_{t,1}^\varepsilon(\rho_1) \leq\left(\int_\Omega \frac{\rho_1}{2(s-t)} \, d\xi_t^\varepsilon + C_2\right).
\end{equation*}
\end{prop}

\begin{proof}
We first compute, by integration by parts and (\ref{eqn: parabolic Allen-Cahn with contact energy}) respectively, if we are given an appropriately smooth placeholder function, $\rho$, that (for a function $f$ we write $f_t$ for $\partial_t f$ throughout this proof for brevity)
\begin{equation*}
    \begin{split}
        \frac{d}{dt} \mu_{t,1}^\varepsilon(\rho) =& \frac{d}{dt}\int_\Omega \rho \, d\mu_{t,1}^\varepsilon \\
        =& \int_\Omega \rho_t \, d\mu_{t,1}^\varepsilon + \int_\Omega \rho \left[\varepsilon 
        \nabla u \cdot \nabla u_t + \frac{W'(u)}{\varepsilon}u_t\right]\\
        =& \int_\Omega \rho_t \, d\mu_{t,1}^\varepsilon + \int_\Omega \rho \left[-\varepsilon u_{t} \Delta u + \frac{W'(u)}{\varepsilon}u_t \right]-\int_\Omega \varepsilon (\nabla \rho \cdot \nabla u ) u_t + \int_{\partial \Omega} 
        \varepsilon \rho (\nabla u \cdot \nu) u_t\\
        =& \int_\Omega \rho_t \, d\mu_{t,1}^\varepsilon - \int_\Omega \varepsilon \rho u_t^2 - \int_\Omega \varepsilon (\nabla \rho \cdot \nabla u) u_t + \int_{\partial \Omega} \varepsilon\rho(\nabla u \cdot \nu) u_t.
    \end{split}
\end{equation*}
We now add and subtract the terms $\varepsilon (\nabla \rho \cdot \nabla u) u_t$ and $\varepsilon \frac{(\nabla \rho \cdot \nabla u)^2}{\rho}$ and continue with the above to see that
\begin{equation*}
    \begin{split}
        \frac{d}{dt}\mu_{t,1}^\varepsilon(\rho) =& \int_\Omega \rho_t d\mu_{t,1}^\varepsilon + \int_\Omega \varepsilon (\nabla \rho \cdot \nabla u) u_t - \int_\Omega \varepsilon\rho \left(u_t + \frac{(\nabla \rho \cdot \nabla u)}{\rho}\right)^2\\
        &+ \int_\Omega \varepsilon \frac{(\nabla \rho \cdot \nabla u)^2}{\rho} + \int_{\partial \Omega} \varepsilon \rho (\nabla u \cdot \nu) u_t.
    \end{split}
\end{equation*}
As an aside we compute, by (\ref{eqn: parabolic Allen-Cahn with contact energy}) and integration by parts respectively, that
\begin{equation*}
    \begin{split}
        \int_\Omega \varepsilon (\nabla \rho \cdot \nabla u) u_t =& \int_\Omega \varepsilon (\nabla \rho \cdot \nabla u) \left(\Delta u - \frac{W'(u)}{\varepsilon^2}\right)\\
        =& -\int_\Omega \varepsilon\rho_i u_{ij}u_j + \rho_{ij}u_i u_j + \int_\Omega \varepsilon \Delta \rho \frac{W(u)}{\varepsilon}\\
        &+\int_{\partial \Omega} \varepsilon(\nabla\rho \cdot \nabla u) (\nabla u \cdot \nu) - \int_{\partial \Omega} \varepsilon (\nabla \rho \cdot \nu) \frac{W(u)}{\varepsilon}.
    \end{split}
\end{equation*}
Observing that
\begin{equation*}
    \begin{split}
        \int_\Omega \rho_iu_{ij}u_j = -\int_\Omega \Delta \rho |\nabla u|^2 - \int_\Omega \rho_i u_{ij}u_j + \int_{\partial\Omega} (\nabla \rho \cdot \nu)|\nabla u|^2,
    \end{split}
\end{equation*}
we have that
\begin{equation*}
    \begin{split}
        \int_\Omega \varepsilon(\nabla \rho \cdot \nabla u) u_t =& \int_\Omega \Delta \rho \left( \frac{\varepsilon|\nabla u|^2}{2} + \frac{W(u)}{\varepsilon} \right) - \int_\Omega \varepsilon \rho_{ij} u_i u_j\\
        &+ \int_{\partial \Omega} \varepsilon (\nabla \rho \cdot \nabla u)(\nabla u \cdot \nu) - \int_{\partial \Omega} (\nabla \rho \cdot \nu) \left(\frac{\varepsilon|\nabla u|^2}{2} + \frac{W(u)}{\varepsilon}\right).
    \end{split}
\end{equation*}
Combining this with our initial calculations above, we have
\begin{equation*}
    \begin{split}
        \frac{d}{dt}\mu_{t,1}^\varepsilon(\rho) =& - \int_\Omega \varepsilon\rho \left(u_t + \frac{(\nabla \rho \cdot \nabla u)}{\rho}\right)^2 + \int_\Omega \varepsilon \frac{(\nabla \rho \cdot \nabla u)^2}{\rho} + \int_\Omega (\rho_t + \Delta \rho )d\mu_{t,1}^\varepsilon - \int_\Omega \varepsilon \rho_{ij}u_iu_j\\
        &+ \int_{\partial \Omega} \varepsilon \rho (\nabla u \cdot \nu) u_t + \int_{\partial \Omega} \varepsilon (\nabla \rho \cdot \nabla u)(\nabla u \cdot \nu) - \int_{\partial \Omega}  (\nabla \rho \cdot \nu) \left(\frac{\varepsilon|\nabla u|^2}{2} + \frac{W(u)}{\varepsilon}\right).
    \end{split}
\end{equation*}
By setting $\rho = \rho_k$ for $k = 1,2$ in the above, noting that $(\nabla(\rho_1 + \rho_2) \cdot \nu) = 0$ on $\partial \Omega$, and dropping the non-positive first term we obtain
\begin{equation*}
    \begin{split}
        \frac{d}{dt} \mu_{t,1}^\varepsilon(\rho_1 + \rho_2) \leq& \sum_{k = 1}^2 \int_\Omega \varepsilon \frac{(\nabla \rho_k \cdot \nabla u)^2}{\rho_k} + ((\rho_k)_t + \Delta \rho_k)\mu_{t,1}^\varepsilon - \varepsilon (\rho_k)_{ij} u_i u_j\\
        &+ 2 \int_{\partial \Omega} \varepsilon \rho_1 (\nabla u \cdot \nu) u_t + \varepsilon (\nabla^\top \rho_1 \cdot \nabla u)(\nabla u \cdot \nu).
    \end{split}
\end{equation*}
Using the fact that $\mu_{t,1}^\varepsilon = \varepsilon |\nabla u|^2 - \xi_t^\varepsilon$ and setting $a^\varepsilon = \frac{\nabla u}{|\nabla u|}$ we then see that, integrating by parts twice to swap indices in the third term on the first line above, we have
\begin{equation}\label{eqn: time derivative of interior energy in monotonicity}
    \begin{split}
        \frac{d}{dt} \mu_{t,1}^\varepsilon(\rho_1 + \rho_2) \leq& \sum_{k = 1}^2 \int_\Omega \varepsilon |\nabla u|^2 \left(\frac{(\nabla \rho_k \cdot a^\varepsilon)^2}{\rho_k} + ((\rho_k)_t + (I - a^\varepsilon \otimes a^\varepsilon)\nabla^2 \rho_k)) \right)\\
        &-\sum_{k = 1}^2\int_\Omega ((\rho_k)_t + \Delta \rho_k) d\xi_t^\varepsilon + 2\int_{\partial \Omega} \varepsilon \rho_1(\nabla u \cdot \nu) u_t + \varepsilon (\nabla^\top \rho_1 \cdot \nabla u)(\nabla u \cdot \nu).
    \end{split}
\end{equation}
Recalling the boundary condition $\varepsilon (\nabla u \cdot \nu)= -\sigma'(u)$ we have
\begin{equation*}
    \begin{split}
        \frac{d}{dt} \mu_{t,1}^\varepsilon(\rho_1 + \rho_2) \leq& \sum_{k = 1}^2 \int_\Omega \varepsilon |\nabla u|^2 \left(\frac{(\nabla \rho_k \cdot a^\varepsilon)^2}{\rho_k} + ((\rho_k)_t + (I - a^\varepsilon \otimes a^\varepsilon)\nabla^2 \rho_k)) \right)\\
        &-\sum_{k = 1}^2\int_\Omega ((\rho_k)_t + \Delta \rho_k) d\xi_t^\varepsilon - 2\int_{\partial \Omega} \sigma'(u)(\rho_1 u_t + (\nabla^\top \rho_1 \cdot \nabla u)).
    \end{split}
\end{equation*}
We also note that
\begin{equation}\label{eqn: time derivative of boundary energy in monotonicity}
    \begin{split}
        \frac{d}{dt}\int_{\partial \Omega} \sigma(u) (\rho_1 + \rho_2) = 2 \frac{d}{dt}\int_{\partial \Omega} \sigma (u) \rho_1 = 2\int_{\partial \Omega} \sigma'(u)u_t\rho_1 + \sigma(u)(\rho_1)_t.
    \end{split}
\end{equation}
Thus, by summing (\ref{eqn: time derivative of interior energy in monotonicity}) and (\ref{eqn: time derivative of boundary energy in monotonicity}) and integrating by parts
\begin{align*}
        \frac{d}{dt} \int_\Omega (\rho_1 + \rho_2) \, d\mu_{t,1}^\varepsilon &+ \frac{d}{dt}\int_{\partial \Omega} \sigma(u) (\rho_1 + \rho_2)\\
        \leq& \sum_{k = 1}^2 \int_\Omega \varepsilon |\nabla u|^2 \left(\frac{(\nabla \rho_k \cdot a^\varepsilon)^2}{\rho_k} + ((\rho_k)_t + (I - a^\varepsilon \otimes a^\varepsilon)\nabla^2 \rho_k)) \right)\\
        &-\sum_{k = 1}^2\int_\Omega ((\rho_k)_t + \Delta \rho_k) d\xi_t^\varepsilon + 2\int_{\partial \Omega} \sigma(u)(\rho_1)_t - \sigma'(u)(\nabla^\top \rho_1 \cdot \nabla u)\\
        =& \sum_{k = 1}^2 \int_\Omega \varepsilon |\nabla u|^2 \left(\frac{(\nabla \rho_k \cdot a^\varepsilon)^2}{\rho_k} + ((\rho_k)_t + (I - a^\varepsilon \otimes a^\varepsilon)\nabla^2 \rho_k)) \right)\\
        &-\sum_{k = 1}^2\int_\Omega ((\rho_k)_t + \Delta \rho_k) d\xi_t^\varepsilon + 2\int_{\partial \Omega} \sigma(u)((\rho_1)_t + \Delta^\top \rho_1).
\end{align*}
Now, we recall the following identity (which follows by a direct computation as in \cite[Section 3.1]{I93})
\begin{equation*}
    \frac{(a \cdot \nabla \rho)^{2}}{\rho} + (( I - a \otimes a) \cdot \nabla^{2} \rho) + \rho_{t} = 0,
\end{equation*}
where $a$ is any unit vector, and note that for $x \in N_{\kappa}$, and $y \in N_{\frac{\kappa}{2}}$, we have
\begin{equation*}
    |x - \zeta(x)| = |\tilde{x} - \zeta(x)| \leq |\tilde{x} - y|.
\end{equation*}
Then, one follows the computations in \cite[Lemma 3.2]{MT15}, to deduce
\begin{equation*}
    \frac{(a \cdot \nabla \tilde{\rho})^{2}}{\tilde{\rho}} + (( I - a \otimes a) \cdot \nabla^{2} \tilde{\rho}) + \tilde{\rho}_{t} \leq C \left( \frac{|\tilde{x} - y|}{s - t} + \frac{|\tilde{x} - y|^{3}}{(s - t)^{2}} \right) \tilde{\rho},
\end{equation*}
for a constant, $C$, depending only on $n$ and $\Omega$.
Continuing from above, and following near identical computations to the that of \cite[Proposition 3.1]{MT15}, we have
\begin{equation}\label{eqn: monotnonicity pre boundary computation}
    \begin{split}
         &\frac{d}{dt} \int_\Omega (\rho_1 + \rho_2) \, d\mu_{t,1}^\varepsilon + \frac{d}{dt}\int_{\partial \Omega} \sigma(u) (\rho_1 + \rho_2)\\
         &\quad\quad\quad\leq \int_\Omega \frac{\rho_1 + \rho_2}{2(s-t)} \, d\xi_t^\varepsilon + \frac{C}{(s-t)^{\frac{3}{4}}}\int_\Omega (\rho_1 + \rho_2) \, d\mu_{t,1}^\varepsilon + C
         + 2\int_{\partial \Omega} \sigma(u)((\rho_1)_t + \Delta^\top \rho_1).
    \end{split}
\end{equation}
We now compute the term $\rho_t + \Delta^\top\rho$ explicitly. For a point $x \in \partial \Omega$ and an orthonormal basis, $\{\tau_i\}_{i = 1}^{n-1}$, of $T_x \partial \Omega$ we have
$$\nabla^\top\rho = \sum_{i = 1}^{n-1} (\nabla \rho \cdot \tau_i) \tau_i = -\sum_{i = 1}^{n-1} \left(\frac{\rho}{2(s-t)}(x-y) \cdot \tau_i \right) \tau_i,$$
and so
$$\Delta^\top\rho = \sum_{j = 1}^{n-1} D_{\tau_j} (\nabla^\top\rho) \cdot \tau_j = \sum_{j = 1}^{n-1} D_{\tau_j}\left(-\sum_{i = 1}^{n-1} \left(\frac{\rho}{2(s-t)}(x-y) \cdot \tau_i \right) \tau_i \right) \cdot \tau_j.$$
Letting $D_{\tau_i}\tau_j = \sum_{k = 1}^{n-1}\Gamma^k_{ij}\tau_k + A_{ij}\nu$ and noting that both $D_{\tau_{i}} (x-y) = \tau_{i}$ and $(\tau_i \cdot \nu) = 0$ for each $1 \leq i \leq n-1$, we rewrite $\Delta^\top \rho$ as 
\begin{align*} \Delta^\top \rho =& - \sum_{j = 1}^{n-1} D_{\tau_j}\left(\frac{\rho}{2(s-t)}(x-y) \cdot \tau_j \right) - \sum_{i,j = 1}^{n-1} \left(\frac{\rho}{2(s-t)}(x-y) \cdot \tau_i \right) D_{\tau_j}\tau_i \cdot \tau_j\\
=& \sum_{j = 1}^{n-1} \left(\frac{\rho}{4(s-t)^2}((x-y) \cdot \tau_j)^2 \right) - \sum_{j = 1}^{n-1} \frac{\rho}{2(s-t)} - \sum_{j = 1}^{n-1} \frac{\rho}{2(s-t)}(x-y) \cdot \left(\sum_{k = 1}^{n-1}\Gamma_{jj}^k\tau_k + A_{jj}\nu\right)\\
&- \sum_{i,j = 1}^{n-1} \left(\frac{\rho}{2(s-t)}(x-y) \cdot \tau_i\right)\Gamma_{ji}^j.
\end{align*}
Reindexing $k$ as $i$ in the above and noting that $\Gamma_{jj}^i + \Gamma_{ji}^j = D_{\tau_j}(\tau_j \cdot \tau_i) = 0$ we see that
\begin{align*}
    \Delta^\top \rho =& \sum_{j = 1}^{n-1} \left(\frac{\rho}{4(s-t)^2}((x-y) \cdot \tau_j)^2 \right) - \frac{(n-1)}{2(s-t)}\rho + \sum_{j = 1}^{n-1} \left(\frac{\rho}{2(s-t)} (x-y) \cdot A_{jj}\nu)\right)\\
    =& \frac{|x-y|^2 - ((x-y)\cdot \nu)^2}{4(s-t)^2}\rho - \frac{(n-1)}{2(s-t)}\rho - \frac{(x-y) \cdot H_{\partial\Omega}}{2(s-t)}\rho,
\end{align*}
where $H_{\partial\Omega} = \sum_{j = 1}^{n-1} A_{jj}\nu$ denotes the mean curvature vector of $\partial \Omega$. Noting that
$$\rho_t = \frac{(n-1)}{2(s-t)} \rho - \frac{|x-y|^2}{4(s-t)^2}\rho,$$
we conclude that
\begin{equation}\label{eqn: kernel on boundary computation}
\rho_t + \Delta^\top\rho = -\frac{((x-y)\cdot \nu)^2}{4(s-t)^2}\rho - \frac{(x-y)\cdot H_{\partial\Omega}}{2(s-t)}\rho.
\end{equation}
By dropping the non-positive first term in (\ref{eqn: kernel on boundary computation}), we get (noting that the kernel is uniformly bounded on the support of the cutoff function, and now allowing our constant, $C$, appearing below to depend on supremum of the mean curvature of $\partial \Omega$) the bound
$$(\rho_1)_t + \Delta^\top\rho_1 \leq C + C\frac{|x-y|}{2(s-t)}\rho_1.$$
Using this bound and by splitting the integral in the same manner as in \cite[(3.15)/(3.16)]{MT15} we conclude that for the boundary term appearing in (\ref{eqn: monotnonicity pre boundary computation}) and for some constants, $C$ and $\widetilde{C}$, depending only on $n, \Omega$ and $E_0$ we have 
$$2\int_{\partial \Omega} \sigma(u)((\rho_1)_t + \Delta^\top \rho_1) \leq C + \frac{\widetilde{C}}{(s-t)^\frac{3}{4}}\int_{\partial\Omega} \sigma(u)(\rho_1 + \rho_2).$$
By considering an integrating factor and incorporating the bounds above into (\ref{eqn: monotnonicity pre boundary computation}) one has
\begin{equation*}
        \frac{d}{dt}\left(e^{C_1(s-t)^\frac{1}{4}}\left(\int_\Omega (\rho_1 + \rho_2) \, d\mu_{t,1}^\varepsilon  + \int_{\partial \Omega} \sigma(u) (\rho_1 + \rho_2)\right)\right)
        \leq e^{C_1(s-t)^\frac{1}{4}}\left(\int_\Omega \frac{\rho_1 + \rho_2}{2(s-t)} \, d\xi_t^\varepsilon + C_2\right),
\end{equation*}
or more succinctly
\begin{equation*}
    \frac{d}{dt}(e^{C_1(s-t)^\frac{1}{4}}\mu_t^\varepsilon(\rho_1 + \rho_2)) \leq e^{C(s-t)^\frac{1}{4}}\left(\int_\Omega \frac{\rho_1 + \rho_2}{2(s-t)} \, d\xi_t^\varepsilon + C_2\right),
\end{equation*}
where $C_1$ and $C_2$ are constants depending only on $n, \Omega$ and $E_0$. Finally, as in \cite[Proposition 3.1]{MT15}, for $s > t > 0$ and points $y \in \Omega \setminus N_\frac{\kappa}{2}$ in the above we ignore both of the boundary terms and set $\rho_2$ to be identically zero, obtaining the interior monotonicity by the estimates for $\rho_1$.
\end{proof}

\section*{Acknowledgements}

KMS was supported in part by the EPSRC [EP/N509577/1], [EP/T517793/1]. YT acknowledges the support of the JSPS Grant-in-aid for scientific research [23H00085]. Part of this work was completed while MW was a PhD student at The EPSRC Centre for Doctoral Training in Geometry and Number Theory (The London School of Geometry and Number Theory), University College London, supported by the EPSRC [EP/S021590/1].

\begingroup

\bibliographystyle{alpha} 
\bibliography{main}
\endgroup

\hrule 

\Addresses

\end{document}